%% file: flat.tex
\documentclass[11pt]{amsart}

\usepackage{import}

\import{./}{Flat-commands.tex}

\calclayout

\newtheorem{theorem}{Theorem}
\newtheorem*{theorem*}{Theorem}
\newtheorem{corollary}[theorem]{Corollary}
\newtheorem{lemma}[theorem]{Lemma}
\newtheorem{proposition}[theorem]{Proposition}

\newtheorem*{question*}{Question}

\newtheorem{innercustomthm}{Theorem}
\newenvironment{customthm}[1]
{\renewcommand\theinnercustomthm{#1}\innercustomthm}
{\endinnercustomthm}

\theoremstyle{definition}
\newtheorem{definition}[theorem]{Definition}
\theoremstyle{remark}

\newtheorem*{remark*}{Remark}

\usepackage{mathrsfs}
\usepackage[
	scr=boondox,
	cal=euler
]{mathalpha}

\usepackage{xcolor}
\definecolor{color-citas}{RGB}{9,21,128}

\usepackage{hyperref}
\hypersetup{
    colorlinks=true,
    linkcolor=color-citas,
    citecolor=color-citas
}
\usepackage{cleveref}
\usepackage{tikz-cd}
\usepackage{enumitem}

\allowdisplaybreaks

\begin{document} 

\renewcommand{\bibname}{References} 

\title[Flat Generalized Connections on Courant Algebroids]{Flat Generalized Connections on \\ Courant Algebroids}
\author{Gil R. Cavalcanti}
\author{Jaime Pedregal}
\address{Department of Mathematics, Utrecht University, 3508 TA Utrecht, The Netherlands}
\email{g.r.cavalcanti@uu.nl, j.pedregalpastor@uu.nl}
\author{Roberto Rubio}
\address{Universitat Aut\`onoma de Barcelona, 08193 Barcelona, Spain}
\email{roberto.rubio@uab.es}

\thanks{This project has been supported by FEDER/AEI/MICINN through the grant PID2022-137667NA-I00 (GENTLE). G.C. and J.P. are also supported by the Open Competition grant number OCENW.M.22.264 from NWO, and R.R. by AGAUR under 2021-SGR-01015 and by FEDER/AEI/MICINN under RYC2020-030114-I}

\begin{abstract}
We consider a family of metric generalized connections on transitive Courant algebroids, which includes the canonical Levi-Civita connection, and study the flatness condition. We find that the building blocks for such flat transitive Courant algebroids are compact simple Lie groups. Further, we give a description of left-invariant flat Levi-Civita generalized connections on such Lie groups, which, in particular, shows the existence of non-flat ones.
\end{abstract}

\maketitle
\tableofcontents

\bigskip
\section{Introduction}

In differential geometry, flat spaces can be seen as the ``simplest'' objects. For example, a flat Riemannian manifold is locally Euclidean not only smoothly, but geometrically. Indeed, a flat Riemannian manifold admits local orthonormal parallel frames, and the torsion-freeness of the Levi-Civita connection implies that these frames come from local coordinates on the manifold. Hence, flat Riemannian manifolds have no local geometric invariants. This reasoning applies to all integrable geometries (those admitting a compatible torsion-free connection, or, more precisely, 1-integrable $G$-structures). For these, flatness means the absence of local geometric invariants.

Non-integrable geometries, however, have also attracted the attention of both mathematicians and physicists. Examples include almost Hermitian, contact metric, naturally reductive spaces... Many of these geometries carry a canonical metric connection with skew-symmetric torsion (see, e.g., \cite{agricola-srni-lectures}), so their study falls under the umbrella of Bismut geometry, by which we mean Riemannian geometry endowed with a 3-form. In this context, the relevant connections carrying the geometrical information of a Bismut manifold $(M, g, H)$ are the two Bismut connections, which are metric and have torsion $\pm H$. Flatness of one of these connections implies the existence of local orthonormal parallel frames, but the presence of torsion now excludes the possibility of these frames being coordinate frames. Instead, these frames and their Lie brackets lead to two possibilities: either the triple is locally isometric to a compact simple Lie group with a bi-invariant metric and the Cartan form, or it isometric to $\Sbb^7$. This result was already known to Cartan and Schouten \cite{cartan-schouten} (see \cite{agricola-friedrich-bismut-flat} for a short classification-free proof of this result).

Bismut connections received renewed interest after Hitchin’s insight \cite{hitchin,gualtieri-branes} that the Bismut connections of a triple $(M,g,H)$, for $g$ a Riemannian metric and $H$ a closed 3-form, were intimately related to the structure of a generalized metric on an exact Courant algebroid. This motivated a whole line of research on Bismut geometry from the generalized geometric viewpoint. Particular examples of this are, to name a couple, the description of bi-Hermitian geometry as generalized Kähler \cite{gualtieri-phd,gualtieri-gK-geom} with all the further research that this motivated \cite{lin-tolman,cavalcanti,hitchin-gK,goto-deformations,bischoff-gualtieri-zabzine,gualtieri-gK-Ham-def,boulanger,goto-scal,apostolov-streets,doppenschmitt-phd} (among many others); and the recasting of the heterotic supergravity equations as a Ricci-flatness condition for the generalized Ricci curvature \cite{coimbra-strickland-constable-waldram,garcia-fernandez-heterotic,coimbra-minasian-triendl-waldram,jurco-moucka-vysoky}. More generally, it has also led to new insights into the Strominger and related systems \cite{garcia-fernandez-rubio-tipler,garcia-fernandez-spinors,garcia-fernandez-gonzalez-molina,garcia-fernandez-gonzalez-molina-streets}, and has given a natural framework for generalized Ricci flow, first introduced in \cite{oliynyk-suneeta-woolgar} (according to \cite{streets}), which has been later rephrased and further developed in generalized-geometric language \cite{severa-valach,garcia-fernandez-spinors,garcia-fernandez-streets,streets-strickland-constable-valach}.

The aim of this paper is to answer the question of flatness in generalized geometry, where the situation seems, in general terms, quite different from the Riemannian and the Bismut situations. First of all, one of the striking features of generalized Riemannian geometry is the non-uniqueness of Levi-Civita connections. Here, we introduce a three-dimensional family of generalized connections, which we call the Bismut family, interpolating between the two canonical generalized connections previously considered in the literature \cite{gualtieri-branes,garcia-fernandez-rubio-tipler}. In this family, there is only one Levi-Civita connection, precisely that of  \cite{garcia-fernandez-rubio-tipler}. Secondly, being generalized-flat does not necessarily mean the existence of local parallel frames, a behavior already observed for Lie algebroids \cite{fernandes-poisson,fernandes-lie-algebroids}. 
We believe this should be linked to a notion of ``generalized holonomy'', but we leave this for future study.

Our main result (phrased for the exact Courant algebroid $\Tbb M$) is the following:

\begin{theorem*}
Let $(M,g,H)$ be a triple with $H\neq 0$, $dH=0$ and $(M,g)$ a 1-connected, complete and irreducible Riemannian manifold, and such that $\Tbb M$ admits a flat generalized Bismut connection. Then $(M,g,H)$ is a compact simple Lie group with a bi-invariant metric and a multiple of the Cartan 3-form.
\end{theorem*}
\noindent For the general statement on transitive Courant algebroids, see Theorem \ref{thm:flat-canonical-all}. It is important to stress here that the proof divides into two different cases. In the first one, whenever the generalized Bismut connection is not Levi-Civita, the triple $(M,g,H)$ will turn out to be Bismut-flat, and our result (Proposition \ref{prop:flat-canonical-non-Levi-Civita}) will follow from the Cartan--Schouten theorem \cite{cartan-schouten,agricola-friedrich-bismut-flat} on Bismut-flat manifolds. In the second case, when the generalized Bismut connection is Levi-Civita, the manifold will not be Bismut-flat a priori, and a new argument is needed, involving the study of parallel 3-forms on symmetric spaces of compact type using the Chevalley--Eilenberg complex (Proposition \ref{prop:flat-canonical-Levi-Civita}).

\bigskip

Since the canonical generalized Levi-Civita connection on compact simple Lie groups is flat, a new question arises naturally: are there any non-flat Levi-Civita connections on such Lie groups? In order to give an answer, we give an explicit description of the space of flat left-invariant generalized Levi-Civita connections for a Lie group $G$ with Lie algebra $\gfrak$, which is seen to be strictly smaller than that of generalized Levi-Civita connections.

\begin{customthm}{\ref{thm:lie-groups}}
The space of flat left-invariant generalized Levi-Civita connections on $\Tbb G$ is an affine space modeled on
\[
(S^2\gfrak^*\otimes\gfrak^*)^G/(S^3\gfrak^*)^G \oplus (S^2\gfrak^*\otimes\gfrak^*)/S^3\gfrak^*.
\]
In particular, there are non-flat left-invariant generalized Levi-Civita connections on $\Tbb G$.
\end{customthm}

The structure of the paper is as follows. In Section \ref{sect:curvature-trans-CA} we explicitly describe the curvature of metric generalized connections with pure-type torsion on transitive Courant algebroids. In Section \ref{sect:canonical-fam}, we introduce the Bismut family of generalized connections and prove Theorem \ref{thm:flat-canonical-all}. Finally, in Section \ref{sect:lie-groups}, we give a proof of Theorem \ref{thm:lie-groups}.

\bigskip
\noindent \textbf{Acknowledgements:} We are grateful for the interesting conversations stemming from talks at the GENTLE seminar given by Ilka Agricola, Vicente Cort\'es and Mario Garc\'ia-Fern\'andez.

\bigskip

\noindent \textbf{Notation and conventions: } We work on the smooth category. On a manifold $M$, vector fields will be denoted by $X,Y,Z,W$ and 1-forms by $\alpha,\beta$.

We will use $\gfrak$ to denote both a Lie algebra or a bundle of Lie algebras. The meaning will be clear from the context.

We will also use the following notation: if $(V_i,\prodesc{\cdot}{\cdot}_i)$, for $i=1,2,3$, are vector spaces (or bundles) with nondegenerate symmetric bilinear pairings, then an element $B\in V_1^*\otimes V_2^*\otimes V_3^*$ will be regarded in three different ways:
\begin{itemize}[label={\tiny$\bullet$}]
	\item as a linear map $V_1\otimes V_2\otimes V_3\to\R$, and we write $B(a_1,a_2,a_3)\in\R$ for the image of $a_i\in V_i$,
	\item as a linear map $V_2\to\Hom(V_1,V_3)$, and we write $B_{a_1}a_2\in V_3$,
	\item as a linear map $V_2\otimes V_3\to V_1$, and we write $B(a_2,a_3)\in V_1$.
\end{itemize}
These three different viewpoints are related by
\begin{equation}\label{eq:the-three-lives-of-B}
B(a_1,a_2,a_3)=\prodesc{B_{a_1}a_2}{a_3}_3=\prodesc{a_1}{B(a_2,a_3)}_1.    
\end{equation}

Finally, operators will always be understood to be applied to the symbol immediately following the operator. For instance, for a form $H\in\Omega^2(M,TM)$ and a connection $\nabla$, the expression $\nabla_XH(Y,Z)$ will always mean the tensor $\nabla_XH$ applied to the pair $(Y,Z)$. If the operator is applied to symbols beyond the one following immediately, then parenthesis will always be used. For example,
\[
\nabla_XH(Y,Z)=\nabla_X(H(Y,Z))-H(\nabla_XY,Z)-H(Y,\nabla_XZ).
\]


\bigskip
\section{Curvature on transitive Courant algebroids} \label{sect:curvature-trans-CA}

We start this section by recalling some standard definitions in the theory of generalized geometry.

\begin{definition}
A \textbf{Courant algebroid} over a manifold $M$ is a vector bundle $E$ endowed with a bundle map $\rho:E\to TM$, called the \textbf{anchor}, a non-degenerate bilinear pairing $\prodesc{\cdot}{\cdot}$ and a bilinear bracket $[\cdot,\cdot]$ on $\Gamma(E)$ satisfying
\begin{enumerate}[label=(\alph*)]
    \item $[a,[b,c]]=[[a,b],c]+[b,[a,c]],$
    \item $\Lcal_{\rho a}\prodesc{b}{c}=\prodesc{[a,b]}{c} + \prodesc{b}{[a,c]}$, \label{item:invariance-of-bracket}
    \item $2[a,a]  = \rho^* d\prodesc{a}{a},$   
\end{enumerate}
where $\rho^*:T^*M\to E^*\cong E$ by using $\prodesc{\cdot}{\cdot}$.
\end{definition}

Sections of a Courant algebroid $E$ will be usually denoted by $a, b, c\in \Gamma(E)$.

From axiom \ref{item:invariance-of-bracket} and the non-degeneracy of the pairing one can deduce a Leibniz rule for the bracket:
\[
[a,fb]=f[a,b]+(\Lcal_{\rho a}f)b,\qquad\text{for $f\in\Cinf(M)$.}
\]

Isomorphisms of Courant algebroids are vector bundle isomorphisms between Courant algebroids compatible with the pairings, the anchor maps and the brackets.
 
In general, we have that $\rho\rho^*=0$, which gives a sequence
\begin{equation} \label{eq:sequence-CAs}
0\longrightarrow T^*M\stackrel{\rho^*}{\longrightarrow} E\stackrel{\rho}{\longrightarrow} TM\longrightarrow 0.
\end{equation}

\begin{definition}
If $\rho$ is surjective, then the Courant algebroid $E$ is called \textbf{transitive}. If moreover the sequence \eqref{eq:sequence-CAs} is exact, then $E$ is called \textbf{exact}.
\end{definition}

Any exact Courant algebroid $E$ is isomorphic to the so-called \textbf{generalized tangent bundle} $\Tbb M:=TM\oplus T^*M$ with the anchor map $\rho(X+\alpha):=X$, the pairing $\prodesc{X+\alpha}{Y+\beta}=\frac{1}{2}(\beta(X)+\alpha(Y))$ and the bracket 
\[ [X+\alpha, Y+\beta] :=[X,Y]+\Lcal_X\beta-i_Yd\alpha + i_Y i_X H, \]
known as the Dorfman bracket, for some closed $H\in \Omega^3(M)$. The form $H$ is not uniquely determined, but can be chosen within its cohomology class, which is known as the \v{S}evera class of $E$.

In the case of transitive Courant algebroids, there is also a standard form, which is more convoluted. Recall that an isotropic splitting of $E$ is a splitting of \eqref{eq:sequence-CAs} whose image is isotropic with respect to $\prodesc{\cdot}{\cdot}$. These always exist on any Courant algebroid.

\begin{proposition}[{\cite{chen-stienon-xu,garcia-fernandez-heterotic}}] \label{prop:transitive-cas-normal-form}
A transitive Courant algebroid $E\to M$ induces a quadratic Lie algebra bundle $(\gfrak\to M,[\cdot,\cdot],\prodesc{\cdot}{\cdot})$, whose sections we denote by $r$ and $s$. A choice of isotropic splitting on $E$ induces
\begin{enumerate}[label=\upshape{(\alph*)}]
    \item a metric connection $\nabla$ on $\gfrak$ making $[\cdot,\cdot]$ parallel, that is, such that
	\[
	\nabla_X[r,s]=[\nabla_Xr,s]+[r,\nabla_Xs],
	\]
	and a form $F\in\Omega^2(M,\gfrak)$ such that
	\[
	R^\nabla(X,Y)r=[F(X,Y),r].
	\]
    \item a form $H\in\Omega^3(M)$ satisfying the Bianchi identity
	\[
	dH-\langle F\wedge F\rangle=0.
	\]
\end{enumerate}
In this case, $E$ is isomorphic to $\Tbb M\oplus\gfrak$ with
\begin{enumerate}[label=\upshape{(\roman*)}]
    \item anchor $\rho(X+\alpha+r)=X$,
	\item pairing $\prodesc{X+\alpha+r}{Y+\beta+s}=\frac{1}{2}(\beta(X)+\alpha(Y))+\prodesc{r}{s}$,
	\item and bracket
	\begin{align}\label{eq:bracket}
	[X+\alpha+r, Y+\beta+s] &=[X,Y]+\Lcal_X\beta-i_Yd\alpha+i_Yi_XH \nonumber \\
    & \qquad\qquad -\prodesc{i_XF}{s}+\prodesc{i_YF}{r}+\prodesc{\nabla r}{s} \\
	&\qquad\qquad  +F(X,Y)+\nabla_Xs-\nabla_Yr+[r,s]. \nonumber
	\end{align}
\end{enumerate}
\end{proposition}

Sections of $\gfrak$ we will typically denote by $r$, $s$ and $t$. The full equivalence between transitive Courant algebroids and the data described in Proposition \ref{prop:transitive-cas-normal-form} is much more involved than in the exact case and it is not relevant for our purposes. 

\begin{remark*}
The analysis that we will do on a transitive algebroid could be carried out on a regular one just by replacing the tangent bundle by the pertinent regular foliation. On general Courant algebroids, it could be carried out leaf-wise.
\end{remark*}

We turn now to the standard notions in generalized Riemannian geometry.

\begin{definition}
A \textbf{generalized metric} on a Courant algebroid $E\to M$ is a bundle map $\Gcal:E\to E$ that is orthogonal with respect to $\prodesc{\cdot}{\cdot}$ and satisfies $\Gcal^2= 1$. It is called \textbf{admissible} if $\rho_+:V_+\to TM$ is an isomorphism.
\end{definition}

A generalized metric $\Gcal$ is completely determined by its $+1$-eigenbundle $V_+\subseteq E$, or its $-1$-eigenbundle $V_-\subseteq E$. On the standard transitive Courant algebroid $\Tbb M\oplus \gfrak$ (as in Proposition \ref{prop:transitive-cas-normal-form}), after suitable gauge transformations (Courant algebroid isomorphisms covering the identity), an admissible metric can always be written in the form \cite{garcia-fernandez-heterotic}
\[
\Gcal(X+\alpha+r):=g^{-1}\alpha+gX-r,
\]
for some pseudo-Riemannian metric $g$ on $M$. Equivalently,
\begin{align*}
V_+ &=\{X+gX: X\in TM\},\\
V_- &=\{X-gX+r: X\in TM, \ r\in\gfrak\}.
\end{align*}
From now on, we focus on this kind of generalized metrics for a given $g$.

For $a\in E$, we denote the projections of $a$ onto $V_\pm$ as follows:
\[
a_\pm:=\frac{1}{2}(1\pm \Gcal)a.
\]
We will also write $X^\pm:=X\pm gX$ and $\alpha^\pm:=g^{-1}\alpha\pm\alpha$ for the lifts of vectors $X$ and 1-forms $\alpha$ on $M$ to elements of $V_\pm$. Notice that we have the relations $\prodesc{X^\pm}{Y^\pm}=\pm g(X,Y)$ and $\prodesc{\alpha^\pm}{\beta^\pm}=\pm g(\alpha,\beta)$.

Notice as well that using $g$ we can now regard $F$ as an element in 
\[
\Omega^1(\gfrak,\so(TM)):=\Gamma(\gfrak^*\otimes\so(TM)).
\]
According to our convention \eqref{eq:the-three-lives-of-B}, we write $F_r X\in\Xfrak(M)$, i.e.
\[
\prodesc{F(X,Y)}{r}=g(F_r X,Y).
\]
Observe that therefore $\prodesc{i_XF}{r}=g(F_r X)$.

As in the case of exact Courant algebroids, the bracket of $\Tbb M\oplus \gfrak$ is intimately related to the Bismut connections of $(g,H)$, given by $\nabla^\pm:=\nabla^g\pm \frac{1}{2}H$, where $\nabla^g$ is the Levi-Civita connection of $g$. The proof can be found in Appendix \ref{app:computations}.

\begin{lemma} \label{lem:brackets-in-regular-ca}
The bracket of $\Tbb M\oplus\gfrak$ and the Bismut connections are related as follows: if $X,Y\in\Xfrak(M)$, then
\begin{align*}
[X^\pm,Y^\mp] &=(\nabla^\mp_XY)^\mp-(\nabla^\pm_YX)^\pm + F(X,Y), \\
[X^\pm,Y^\pm] &=((\nabla^\pm X)^*Y+[X,Y])^\pm -((\nabla^\pm X)^*Y)^\mp+F(X,Y), \\
[X^\pm,r] &=-\prodesc{i_XF}{r}+\nabla_Xr,
\end{align*}
where the adjoint $^*$ is taken with respect to $g$.
\end{lemma}

Another key notion in generalized Riemannian geometry is that of generalized connections and their torsion.

\begin{definition}
A \textbf{generalized connection} on a Courant algebroid $E$ is a linear operator $D:\Gamma(E)\to\Gamma(E^*\otimes E)$ such that
\begin{align*}
    D(fa)&=\rho^*df\otimes a+fDa,\\
    \rho^*d\prodesc{a}{b}&= (Da)^*b+(Db)^*a.
\end{align*}

The \textbf{torsion} of $D$ is the tensor field $\mathcal{T}\in\Gamma(\wedge^3 E^*)$ given, as a map $\Gamma(\wedge^2 E)\to \Gamma(E)$, by
\[
\mathcal{T}(a,b):=D_ab-D_ba-[a,b]+(Da)^*b.
\]

If $\Gcal$ is a generalized metric on $E$, then $D$ is said to be metric-compatible, or just \textbf{metric}, if $D\Gcal=0$, i.e., if $D_a(\Gcal b)=\Gcal(D_ab)$. Equivalently, if $D_a(\Gamma(V_\pm))\subseteq\Gamma(V_\pm)$, for all $a\in\Gamma(E)$. Moreover, the torsion of $D$ is said to be of \textbf{pure type} if $\mathcal{T}\in\Gamma(\wedge^3 V_+ \oplus \wedge ^3 V_-)$. 
\end{definition}

It is also well-known that the torsion of a metric generalized connection is of pure-type if and only if \cite[Lem. 3.2]{garcia-fernandez-spinors}
\begin{equation}\label{eq:pure-torsion}
D_{a_\mp}b_\pm=[a_\mp,b_\pm]_\pm.    
\end{equation}

We now characterize metric generalized connections with pure-type torsion on the transitive Courant algebroid $\Tbb M\oplus\gfrak$. To express the torsion-freeness condition, let us introduce some further notation.

\begin{definition}
For $V$ a vector space (or bundle), we define the \textbf{Bianchator} map $\Bi:\wedge^2V^*\otimes V^*\to \wedge^3V^*$ by
\[
\Bi(B)(v_1,v_2,v_3):=B(v_1,v_2,v_3)+B(v_2,v_3,v_1)+B(v_3,v_2,v_1).
\]
We also also define its extension $\Bi$ to $\wedge^2V^*\otimes V^{*\otimes k}$ by acting on the first three components. 
\end{definition} 

\begin{definition}
The \textbf{Cartan 3-form} of the quadratic Lie algebra bundle $\gfrak$ is the form $c\in\Gamma(\wedge^3\gfrak^*)$ given by $c:=\prodesc{[\cdot,\cdot]}{\cdot}$. 
\end{definition}

We repeatedly make use of our convention \eqref{eq:the-three-lives-of-B} without mention.

\begin{proposition} \label{prop:connection-tensors}
A metric generalized connection $D$ with pure-type torsion on $\Tbb M\oplus\gfrak$ is given by the following data:
\begin{align*}
  	A&\in\Omega^1(\gfrak,\End TM),& 
	B^\pm&\in\Omega^1(M,\so(TM)),&
	C&\in\Omega^1(M,\End\gfrak), \\
	L&\in\Omega^1(\gfrak, \so(\gfrak)),&
	W&\in\Omega^1(\gfrak, \so(TM)),&
	N&\in\Omega^1(M,\so(\gfrak)).  
\end{align*}
The connection $D$ is then given by
\begin{enumerate}[label=\upshape{(\roman*)}]
    \item $D_{X^\mp}Y^\pm = (\nabla^\pm_XY)^\pm+\frac{1\mp 1}{2}F(X,Y)$, \label{item:D-mp-pm}
    \item $D_{X^\pm}Y^\pm = (\nabla^\pm_XY+B^\pm_XY)^\pm+\frac{1\mp 1}{2}A(X,Y)$, \label{item:D-pm-pm}
    \item $D_{X^+}r =-\frac{1}{2}(F_r X)^-+\nabla_Xr$, \label{item:D-p-g}
    \item $D_{X^-}r =(A_r X)^-+\nabla_Xr+N_Xr$, \label{item:D-m-g}
    \item $D_r X^+ = \frac{1}{2}(F_r X)^+$, \label{item:D-g-p}
    \item $D_r X^- =(W_r X)^-+C_Xr$, \label{item:D-g-m}
    \item $D_r s =C(r,s)^-+L_rs$. \label{item:D-g-g}
\end{enumerate}

Moreover, $D$ is torsion-free if and only if the following conditions are satisfied:
\begin{enumerate}[resume,label=\upshape{(\roman*)}]
	\item $\Bi(B^\pm)=\mp H$, \label{item:torsion-pm-pm-pm}
	\item $W(X,Y)=A(X,Y)-A(Y,X)-F(X,Y)$,
	\item $N(r,s)=C(r,s)-C(s,r)$,
	\item $\Bi(L)=c$. \label{item:torsion-g-g-g}
\end{enumerate}
\end{proposition}

\begin{proof}
Items \ref{item:D-mp-pm} and \ref{item:D-p-g} follow directly from Lemma \ref{lem:brackets-in-regular-ca}, using that $D$ has pure-type torsion:
\begin{align*}
D_{X^\mp}Y^\pm &=[X^\mp,Y^\pm]_\pm=(\nabla^\pm_XY)^\pm+\frac{1\mp 1}{2}F(X,Y), \\
D_{X^+}r &=[X^+,r]_-=-\frac{1}{2}(F_rX)^-+\nabla_Xr. 
\end{align*}

For \ref{item:D-pm-pm}, define two operators $\nablabarpm$ by $\nablabarpm_XY:=\rho(D_{X^\pm}Y^\pm)$. These are metric connections on $TM$, as can be readily checked in a similar manner as in the proof of Lemma \ref{lem:brackets-in-regular-ca} for $\nabla^\pm$. Hence, there are $B^\pm\in\Omega^1(M,\so(TM))$ such that $\nablabarpm=\nabla^\pm+B^\pm$. Since $V_+\cong TM$, we have that $D_{X^+}Y^+=(\nablabar\hspace{0pt}^+_XY)^+$. On the other hand, there is some tensor $A\in\Omega^1(\gfrak,\End TM)$ such that
\[
D_{X^-}Y^-=(\nablabar\hspace{0pt}^-_XY)^-+A(X,Y).
\]

To prove \ref{item:D-m-g}, notice that the expression $\rho(D_{X^-}r)$ is $\Cinf(M)$-linear in both $X$ and $r$, since $\rho r=0$, so there is a tensor $Q\in\Omega^1(\gfrak,\End TM)$ and a connection $\nabla'$ on $\gfrak$ such that
\[
D_{X^-}r=(Q_r X)^-+\nabla'_Xr.
\]
Since $D$ is compatible with the pairing of $E$, we have that, using \ref{item:D-pm-pm},
\[
0=\Lcal_{\rho X^-}\prodesc{Y^-}{r}=\prodesc{A(X,Y)}{r}-g(Y,Q_r X)=g(Y,A_r X-Q_r X),
\]
so that $Q=A$. On the other hand,
\[
\Lcal_X\prodesc{r}{s}=\Lcal_{\rho X^-}\prodesc{r}{s}=\prodesc{\nabla'_Xr}{s}+\prodesc{r}{\nabla'_Xs},
\]
so $\nabla'$ is a metric connection. Hence, there is a tensor $N\in\Omega^1(M,\so(\gfrak))$ such that $\nabla'_Xr=\nabla_Xr +N_X r$.

Turning now to \ref{item:D-g-p}, \ref{item:D-g-m} and \ref{item:D-g-g}, notice that for all $a\in\Gamma(E)$ the expression $D_r a$ is $\Cinf(M)$-linear in $a$, since $\rho r=0$, so it is actually a tensor. Notice first that
\[
D_r X^+=[r,X^+]_+=\frac{1}{2}(F_Xr)^+,
\]
by Lemma \ref{lem:brackets-in-regular-ca}, which settles \ref{item:D-g-p}. Secondly, there are tensors $W\in\Omega^1(\gfrak, \End TM)$ and $C\in\Omega^1(M, \End\gfrak)$ such that
\[
D_r X^-=(W_r X)^-+C_Xr.
\]
Furthermore, since $D$ is compatible with the pairing,
\[
0=\Lcal_{\rho r}\prodesc{X^-}{Y^-}=-g(W_r X,Y)-g(X,W_r Y),
\]
so that actually $W\in\Omega^1(\gfrak,\so(TM))$, giving \ref{item:D-g-m}. Finally, for \ref{item:D-g-g}, there are also tensors $P\in\Omega^1(M,\End\gfrak)$ and $L\in\Omega^1(\gfrak,\End\gfrak)$ such that
\[
D_rs=P(r,s)^-+L_rs.
\]
But observe that
\[
0=\Lcal_{\rho r}\prodesc{X^-}{s}=\prodesc{C_Xr}{s}-g(X,P(r,s))=g(X,C(r,s)-P(r,s)),
\]
so that actually $P=C$. Also,
\[
0=\Lcal_{\rho r}\prodesc{s}{t}=\prodesc{L_rs}{t}+\prodesc{s}{L_rt},
\]
so that in fact $L\in\Omega^1(\gfrak,\so(\gfrak))$, finally establishing \ref{item:D-g-g}.

For \ref{item:torsion-pm-pm-pm}-\ref{item:torsion-g-g-g}, we compute now the torsion of $D$ in terms of all this data, making use of Lemma \ref{lem:brackets-in-regular-ca} throughout the computations:
\begin{align*}
\mathcal{T}^D(X^\pm,Y^\pm,Z^\pm) &= \prodesc{D_{X^\pm}Y^\pm-D_{Y^\pm}X^\pm-[X^\pm,Y^\pm]}{Z^\pm}+\prodesc{D_{Z^\pm}X^\pm}{Y^\pm} \\
&= \pm g(\nablabar^\pm_XY-\nablabar^\pm_YX-(\nabla^\pm X)^*Y-[X,Y],Z)\pm g(\nablabar^\pm_ZX,Y)\\
&=H(X,Y,Z)\pm B^\pm(X,Y,Z)\mp B^\pm(Y,X,Z)\pm B^\pm(Z,X,Y) \\
&=(H\pm \Bi(B^\pm))(X,Y,Z) \\
\mathcal{T}^D(X^-,Y^-,r) &=\prodesc{A(X,Y)-A(Y,X)-[X^-,Y^-]}{r}-g(W_r X,Y) \\
&=\prodesc{A(X,Y)-A(Y,X)-F(X,Y)-W(X,Y)}{r}; \\
\mathcal{T}^D(X^-,r,s) &=\prodesc{\nabla'_Xr-C_Xr-[X^-,r]}{s}+\prodesc{C_Xs}{r} \\
&=\prodesc{N_Xr-C_Xr}{s}+\prodesc{C_Xs}{r} \\
&= g(N(r,s)-C(r,s)+C(s,r), X)\\
\mathcal{T}^D(r,s,t) &=\prodesc{L_rs-L_s r-[r,s]}{t}+\prodesc{L_t r}{s} \\
&=L(r,s,t)-L(s,r,t)+L(t,r,s)-c(r,s,t) \\
&=(\Bi(L)-c)(r,s,t). \qedhere
\end{align*}
\end{proof}

Lastly, we recall the notion of generalized curvature, in the sense of \cite{garcia-fernandez-spinors} (see also \cite{CPR:2024}), which will play a key role in the rest of the paper.

\begin{definition}
Let $D$ be a generalized connection on $E$ and $\Gcal$ a generalized metric. The \textbf{generalized curvatures} of the pair $(D,\Gcal)$ are the operators $\Rcal^\pm_D\in\Gamma(V_\pm^*\otimes V_\mp^*\otimes\so(V_\pm))$ given by
\[
\Rcal^\pm_D(a_\pm,b_\mp):=\Rcal_0(a_\pm,b_\mp),\qedhere
\]
which correspond to the tensorial parts of the na\"ive curvature $\Rcal_0:\Gamma(E)^2\to\Gamma(\so(E))$, defined by
\begin{equation*}
\Rcal_0(a,b)c:=D_aD_bc-D_bD_ac-D_{[a,b]}c.     
\end{equation*}
\end{definition}

Using Proposition \ref{prop:connection-tensors}, we can now explicitly compute the generalized curvatures of metric generalized connections with pure-type torsion. The interested reader can find the computations in Appendix \ref{app:computations}.

\begin{proposition} \label{prop:curvatures-regular-CA}
Let $D$ be a metric generalized connection with pure-type torsion on $\Tbb M\oplus\gfrak$ defined by the tensors in Proposition \ref{prop:connection-tensors}. Then its generalized curvatures $\Rcal^\pm_D$ are given by
\begin{align*}
\Rcal^+_D(X^+,Y^-)Z^+ &= \left( R^+(X,Y)Z-(\nabla^+_YB^+)_XZ -\frac{1}{2} F_{F(X,Y)}Z \right)^+, \\
\Rcal^+_D(X^+,r)Z^+ &= \frac{1}{2}\left( (\nabla^+_XF)_r Z+B^+_X(F_r Z)-F_r B^+_XZ+B^+_{F_r X}Z \right)^+, \\
\Rcal^-_D(X^-,Y^+)Z^- &= \left(R^-(X,Y)Z-(\nabla^-_YB^-)_XZ + A_{F(Y,Z)}X\right. \\
&\qquad\qquad \left. -W_{F(X,Y)}Z+\frac{1}{2}F_{A(X,Z)}Y \right)^- \\
&\qquad\qquad -\nabla^-_YA(X,Z)+\nabla^-_XF(Y,Z)-F(Y,B^-_XZ)\\
&\qquad\qquad -F(H(X,Y),Z)-C_Z(F(X,Y)),\\
\Rcal^-_D(X^-,Y^+)r &=\left(-\frac{1}{2}(\nabla^-_XF)_r Y-(\nabla^-_YA)_r X-\frac{1}{2}B^-_XF_r Y\right. \\
&\qquad\qquad\left.+\frac{1}{2}F_r(H(X,Y))-C(F(X,Y),r) +\frac{1}{2}F_{N_Xr}Y\right)^- \\
&\qquad\qquad +R^\nabla(X,Y)r -(\nabla^-_YN)_Xr-\frac{1}{2}A(X,F_r Y)\\
&\qquad\qquad-F(Y,A_r X)+L_{F(X,Y)}r, \\
\Rcal^-_D(r, Y^+)Z^- &=\left(-(\nabla^-_YW)_r Z+C(r,F(Y,Z))+\frac{1}{2}F_{C_Zr}Y-\frac{1}{2}B^-_{F_r Y}Z
\right)^- \\
&\qquad\qquad -(\nabla^-_YC)_Zr+L_r(F(Y,Z))-F(Y,W_r Z)\\
&\qquad\qquad +\frac{1}{2}A(F_r Y, Z)-\frac{1}{2}F(F_r Y,Z), \\
\Rcal^-_D(r,Y^+)s &=\left( -\nabla^-_YC(r,s)-\frac{1}{2}W_r F_s Y+\frac{1}{2}F_{L_rs}Y \right. \\
&\qquad\qquad \left. +\frac{1}{4}F_s F_r Y +\frac{1}{2}A_s F_r Y\right)^- \\
&\qquad\qquad -(\nabla_YL)_rs -\frac{1}{2}C_{F_s Y}r-F(Y,C(r,s))+N_{F_r Y}s,
\end{align*}
where $R^\pm$ denote the Bismut curvatures.
\end{proposition}

\bigskip
\section{The Bismut family} \label{sect:canonical-fam}

Although there are many generalized (Levi-Civita) connections on general Courant algebroids, in the case of exact and transitive Courant algebroids (in standard form) there are some which involve no choices and hence deserve the name canonical, e.g., the canonical Levi-Civita connection of \cite{garcia-fernandez-rubio-tipler} and the Gualtieri--Bismut connection of \cite{gualtieri-branes}. We focus now on a family that interpolates between these two.

\begin{definition} \label{def:bismut-family}
The \textbf{Bismut family of generalized connections} is the family of metric generalized connections with pure-type torsion on $\Tbb M\oplus\gfrak$ given by the tensors (cfr. Proposition \ref{prop:connection-tensors})
\[
A=C=N=0,\qquad W=-F,\qquad B^\pm=\lambda^\pm H,\qquad L=\mu c,
\]
for some constants $\lambda^\pm,\mu\in\R$. Connections in the Bismut family will be called \textbf{generalized Bismut connections}.
\end{definition}

Notice that, by Proposition \ref{prop:connection-tensors}, of all the generalized Bismut connections, only the one with $\lambda^\pm=\mp\mu=\mp\frac{1}{3}$ is Levi-Civita. This connection has been called the \textbf{canonical Levi-Civita connection} of $(E,\Gcal)$ in \cite{garcia-fernandez-rubio-tipler}. In the case of an exact Courant algebroid, the Bismut family is a plane of generalized connections interpolating between the Gualtieri--Bismut connection \cite{gualtieri-branes}, for $\lambda^\pm=0$, and the canonical Levi-Civita connection, for $\lambda^\pm=\mp\frac{1}{3}$. Back to transitive Courant algebroids, we call the case $\lambda^\pm=\mu=0$ the Gualtieri--Bismut connection. Its torsion is given by
\begin{align*}
\mathcal{T}(X^\pm,Y^\pm,Z^\pm)&=H(X,Y,Z),\\
\mathcal{T}(r,s,t)&=c(r,s,t),
\end{align*}
and zero otherwise.

\begin{proposition} \label{prop:curvatures-generalized-Bismut}
The generalized curvatures $\Rcal^\pm_D$ of a generalized Bismut connection are given by
\begin{align}
\Rcal^+_D(X^+,Y^-)Z^+ &= \left( R^+(X,Y)Z-\lambda^+ \nabla^+_YH(X,Z) -\frac{1}{2} F_{F(X,Y)}Z \right)^+,  \label{eq:curvaturep-XYZ} \\
\Rcal^+_D(X^+,r)Z^+ &= \frac{1}{2}\left( (\nabla^+_XF)_r Z+\lambda^+ H(X,F_r Z)-\lambda^+ F_r (H(X,Z))\right. \label{eq:curvaturep-XxiZ} \\
&\qquad\qquad \left. +\lambda^+H(F_r X,Z) \right)^+,  \nonumber \\
\Rcal^-_D(X^-,Y^+)Z^- &= \left(R^-(X,Y)Z-\lambda^- \nabla^-_YH(X,Z) +F_{F(X,Y)}Z \right)^-  \label{eq:curvaturem-XYZ} \\
&\qquad +\nabla^-_XF(Y,Z)-\lambda^- F(Y,H(X,Z))-F(H(X,Y),Z), \nonumber \\
\Rcal^-_D(X^-,Y^+)r &=\left(-\frac{1}{2}(\nabla^-_XF)_r Y-\frac{\lambda^-}{2}H(X,F_r Y)+\frac{1}{2}F_r(H(X,Y)) \right)^-  \label{eq:curvaturem-XYxi} \\
&\qquad +(\mu+1)[F(X,Y),r], \nonumber \\
\Rcal^-_D(r, Y^+)Z^- &=\left((\nabla^-_YF)_r Z-\frac{\lambda^-}{2}H(F_r Y,Z) \right)^- \label{eq:curvaturem-xiYZ}\\
&\qquad\qquad +\mu[r, F(Y,Z)]+F(Y,F_r Z)-\frac{1}{2}F(F_r Y,Z), \nonumber \\
\Rcal^-_D(r,Y^+)s &=\left( \frac{1}{2}F_r F_s Y+\frac{\mu}{2}F_{[r,s]}Y+\frac{1}{4}F_s F_r Y \right)^-  \label{eq:curvaturem-xiYeta}
\end{align}
\end{proposition}

\begin{proof}
Substitute in Proposition \ref{prop:curvatures-regular-CA} the tensors for a generalized Bismut connection. For \eqref{eq:curvaturem-xiYeta}, take into account that $\nabla c=0$, since $\nabla$ preserves $\prodesc{\cdot}{\cdot}$ and $[\cdot,\cdot]$
\end{proof}

We finally address the question of flat connections on $\Tbb M\oplus\gfrak$. More concretely, the question is: What tuples $(M,g,H,\gfrak,\nabla,F)$ admit a flat generalized Bismut connection?

The answer (Theorem \ref{thm:flat-canonical-all}) will be intimately related to the Bismut geometry of $(M,g,H)$. For this reason, Lemma \ref{lem:bismut-vs-riemann-curvature} below, relating Bismut and Riemannian geometry, will be used repeatedly.

Before that, let us introduce some more helpful notation.

\begin{definition}
If $V$ is a vector space (or bundle), we define the \textbf{skew-symmetrization map} $\Alt_{12}:V^*\otimes V^*\to\wedge^2V^*$ by
\[
\Alt_{12}(B)(v_1,v_2):=\frac{1}{2}(B(v_1,v_2)-B(v_2,v_1)),
\]
and we also denote by $\Alt_{12}$ its extension to higher order tensors.
\end{definition}

\begin{definition}
We define the squares of $H$ and of $F$ as the forms $H^2$ and $F^2$ given by
\begin{align*}
H^2(X,Y,Z,W) &:=g(H(X,Y),H(Z,W)), \\
F^2(X,Y,Z,W) &:=\prodesc{F(X,Y)}{F(Z,W)},
\end{align*}
for $X,Y,Z,W\in TM$.
\end{definition}

Notice that both $H^2$ and $F^2$ lie in $\Gamma(S^2\wedge^2T^*M)$. It is a straightforward computation to check that $\langle F\wedge F\rangle=2\Bi(F^2)$.

\begin{lemma} \label{lem:bismut-vs-riemann-curvature}
For a triple $(M,g,H)$, the $\nabla^\pm$-derivative of $H$ and the Bismut curvatures $R^\pm$ are given by
\begin{enumerate}[label=\upshape{(\roman*)}]
    \item $\nabla^\pm H=\nabla^gH\mp\frac{1}{2}\Bi(H^2),$ \label{lem:bismut-vs-riemann-curvature:nablapmH}
    \item $R^\pm=R^g\pm\Alt_{12}(\nabla^g H)+\frac{1}{4}(H^2-\Bi(H^2)).$ \label{lem:bismut-vs-riemann-curvature:Rpm-vs-Rg}
\end{enumerate}
In particular,
\begin{enumerate}[resume,label=\upshape{(\roman*)}]
    \item $g(R^+(X,Y)Z,W)=g(R^-(Z,W)X,Y)+\frac{1}{2}dH(X,Y,Z,W).$ \label{lem:bismut-vs-riemann-curvature:RBismut-sym}
\end{enumerate}
\end{lemma}

\begin{proof}
For \ref{lem:bismut-vs-riemann-curvature:nablapmH}, we compute
\begin{align*}
\nabla^\pm_XH(Y,Z) &=\nabla^\pm_X(H(Y,Z))-H(\nabla^\pm_XY,Z)-H(Y,\nabla^\pm_XZ) \\
&=\nabla^g_XH(Y,Z)\pm\frac{1}{2}(H(X,H(Y,Z))-H(H(X,Y),Z)-H(Y,H(X,Z))).
\end{align*}

As for \ref{lem:bismut-vs-riemann-curvature:Rpm-vs-Rg}, since $[X,Y]=\nabla^g_XY-\nabla^g_YX$, we have that
\begin{align*}
R^\pm(X,Y)Z &= \nabla^\pm_X(\nabla^g_YZ\pm\frac{1}{2}H(Y,Z))-\nabla^\pm_Y(\nabla^g_XZ\pm\frac{1}{2}H(X,Z)) \\
&\qquad\qquad -\nabla^g_{[X,Y]}Z\mp\frac{1}{2}H([X,Y],Z) \\
&=R^g(X,Y)Z\pm\frac{1}{2}H(X,\nabla^g_YZ)\pm\frac{1}{2}\nabla^g_X(H(Y,Z))\\
&\qquad\qquad +\frac{1}{4}H(X,H(Y,Z))\mp\frac{1}{2}H(Y,\nabla^g_XZ)\mp\frac{1}{2}\nabla^g_Y(H(X,Z))\\
&\qquad\qquad -\frac{1}{4}H(Y,H(X,Z))\mp\frac{1}{2}H([X,Y],Z) \\
&=R^g(X,Y)Z\pm\frac{1}{2}(\nabla^g_XH(Y,Z)-\nabla^g_YH(X,Z))\\
&\qquad\qquad +\frac{1}{4}(H(X,H(Y,Z))+H(Y,H(Z,X))),
\end{align*}

Finally, since $R^g$, $H^2$ and $\Bi(H^2)$ are all sections of $S^2\wedge^2T^*M$, so that we have, for instance, that
\[
H^2(X,Y,Z,W)=H^2(Z,W,X,Y),
\]
then, by \ref{lem:bismut-vs-riemann-curvature:Rpm-vs-Rg},
\begin{align*}
g(R^+(X,Y)Z,W)-g(R^-(Z,W)X,Y) &=\frac{1}{2}\big( \nabla^g_XH(Y,Z,W)-\nabla^g_YH(X,Z,W) \\
&\qquad\qquad +\nabla^g_ZH(X,Y,W)-\nabla^g_WH(X,Y,Z) \big) \\
&=\frac{1}{2}dH(X,Y,Z,W). \qedhere
\end{align*}
\end{proof}

We can now formulate our first main result, which gives the conditions that a tuple $(M,g,H,\gfrak,\nabla,F)$ must satisfy in order to admit a flat generalized Bismut connection. These conditions, concerning mainly the curvatures of $\nabla^g$, $\nabla^\pm$ and $\nabla$ and the parallelism property of $H$, are what will allow us to deduce our main theorem (Theorem \ref{thm:flat-canonical-all}). For the sake of the exposition, we limit ourselves here to the Riemannian case. The remarkably more complicated result for pseudo-Riemannian metrics can be found in Appendix \ref{app:pseudo-riemannian}.

\begin{proposition} \label{prop:generalized-Riemann-flat}
Let $g$ be Riemannian. If $\Tbb M\oplus\gfrak$ admits a flat generalized Bismut connection $D$, then
\begin{enumerate}[label=\upshape{(\roman*)}]
    \item $F=0$, and, in particular, $\nabla$ is flat, \label{prop:generalized-Riemann-flat:F-zero}
    \item $\nabla^gH=0$, \label{prop:generalized-Riemann-flat:H-parallel:Riemannian}
    \item if $(\lambda^+,\lambda^-)\neq (-\frac{1}{3},\frac{1}{3})$, then $\Bi(H^2)=0$, \label{prop:generalized-Riemann-flat:BiH2:Riemannian}
    \item $R^\pm=-\frac{1}{6}\Bi(H^2)$, \label{prop:generalized-Riemann-flat:Rpm:Riemannian}
	\item $R^g=\frac{1}{12}\Bi(H^2)-\frac{1}{4}H^2$, \label{prop:generalized-Riemann-flat:Rg:Riemannian}
	\item $\sec^g(X,Y)=\displaystyle \frac{1}{4}\frac{\norm{H(X,Y)}_g^2}{\norm{X\wedge Y}_g^2}$, \label{prop:generalized-Riemann-flat:sec:Riemannian}
	\item $\scal^g=\frac{1}{4}\norm{H}_g^2$, which is locally constant. \label{prop:generalized-Riemann-flat:scal:Riemannian}
\end{enumerate}

In particular, if $(\lambda^+,\lambda^-)\neq (-\frac{1}{3},\frac{1}{3})$, which implies that $D$ is not generalized Levi-Civita, then $R^\pm=0$ and $R^g=-\frac{1}{4}H^2$.
\end{proposition}

\begin{proof}
First of all, from \eqref{eq:curvaturem-xiYeta} we have that
\[
\mu F_{[r,s]}=-F_r F_s-\frac{1}{2}F_s F_r.
\]
Since the left-hand side is skew-symmetric in $r$ and $s$, then the right-hand side must be so as well. Therefore,
\[
-F_r F_s-\frac{1}{2}F_s F_r=F_s F_r+\frac{1}{2}F_r F_s,
\]
which leads to $F_r F_s=-F_s F_r$. This in turn implies that
\[
\mu F_{[r,s]}=-\frac{1}{2}F_r F_s.
\]
In particular, $F_r F_r=0$, which gives that, since $F_r$ is skew-symmetric,
\[
g(F_r X,F_r X)=-g(F_rF_rX,X)=0.
\]
Since $g$ is Riemannian, this implies that $F=0$. Since the curvature of $\nabla$ is given by $R^\nabla(X,Y)=[F(X,Y),\cdot]$, then $\nabla$ is flat, settling \ref{prop:generalized-Riemann-flat:F-zero}.

Consider now \eqref{eq:curvaturep-XYZ} and \eqref{eq:curvaturem-XYZ}. We have that, since $F=0$,
\begin{equation}\label{eq:nablapm-Rpm:Riemannian}
\lambda^\pm\nabla^\pm_XH(Y,Z)=-R^\pm(X,Y)Z.
\end{equation}
We can rewrite this as
\begin{equation}\label{eq:nablapm-Rpm-noXYZ:Riemannian}
R^\pm=-\lambda^\pm\nabla^\pm H.
\end{equation}
The first easy consequence of \eqref{eq:nablapm-Rpm:Riemannian} is that, since the right-hand side is skew-symmetric in $X$ and $Y$, then the left-hand side must be so as well. If $\lambda^+\neq 0$ or $\lambda^-\neq 0$, then this means that either $\nabla^+H$ or $\nabla^-H$ actually lies in $\Omega^4(M)$, and Lemma \ref{lem:bismut-vs-riemann-curvature}\ref{lem:bismut-vs-riemann-curvature:nablapmH} then implies that $\nabla^gH\in\Omega^4(M)$ as well, since $\Bi(H^2)\in\Omega^4(M)$. But then, because the de Rham differential can be expressed as the skew-symmetrization of $\nabla^g$, we get that $\nabla^gH=\frac{1}{4}dH=\frac{1}{2}\Bi(F^2)=0$, proving \ref{prop:generalized-Riemann-flat:H-parallel:Riemannian}. We conclude from \eqref{eq:nablapm-Rpm-noXYZ:Riemannian} and Lemma \ref{lem:bismut-vs-riemann-curvature}\ref{lem:bismut-vs-riemann-curvature:nablapmH} that
\begin{equation} \label{eq:Rpm:Riemannian}
R^\pm = -\lambda^\pm\nabla^\pm H=\pm\frac{\lambda^\pm}{2}\Bi(H^2).
\end{equation}

The (first) Bianchi identity for $R^\pm$ in general states that
\begin{equation} \label{eq:bianchi-identity-Rpm:Riemannian}
\Bi(R^\pm)=\Bi(H^2)\pm\Bi(\nabla^\pm H).
\end{equation}
Using \eqref{eq:Rpm:Riemannian}, Lemma \ref{lem:bismut-vs-riemann-curvature}\ref{lem:bismut-vs-riemann-curvature:nablapmH} and the fact that $\nabla^gH=0$, we now have that 
\[
\pm\frac{3}{2}\lambda^\pm\Bi(H^2)=\Bi(R^\pm)=\Bi(H^2)\pm\Bi(\nabla^\pm H)=\Bi(H^2)-\frac{3}{2}\Bi(H^2).
\]
This gives the equations
\begin{align*}
(1+3\lambda^+)\Bi(H^2) &=0, \\
(1-3\lambda^-)\Bi(H^2) &=0.
\end{align*}
If $(\lambda^+,\lambda^-)\neq (-\frac{1}{3},\frac{1}{3})$, then one of these equations will give that $\Bi(H^2)=0$, settling \ref{prop:generalized-Riemann-flat:BiH2:Riemannian}. Notice that, by \eqref{eq:Rpm:Riemannian}, this also implies \ref{prop:generalized-Riemann-flat:Rpm:Riemannian}. From Lemma \ref{lem:bismut-vs-riemann-curvature}\ref{lem:bismut-vs-riemann-curvature:Rpm-vs-Rg} it now follows that $R^g=\frac{1}{12}\Bi(H^2)-\frac{1}{4}H^2$, giving \ref{prop:generalized-Riemann-flat:Rg:Riemannian}.

In the case that $\lambda^\pm=0$, we do not know $\nabla^gH$ to be a 4-form and the previous argument does not apply. In this case, though, $R^\pm=0$ directly holds, by \eqref{eq:nablapm-Rpm-noXYZ:Riemannian}. The Bianchi identity \eqref{eq:bianchi-identity-Rpm:Riemannian} gives
\[
0=\Bi(H^2)\pm \Bi(\nabla^\pm H).
\]
Using Lemma \ref{lem:bismut-vs-riemann-curvature}\ref{lem:bismut-vs-riemann-curvature:nablapmH}, we obtain
\[
0=\Bi(H^2)\pm\Bi(\nabla^gH)-\frac{1}{2}\Bi(\Bi(H^2))=\pm\Bi(\nabla^gH)-\frac{1}{2}\Bi(H^2).
\]
This implies that $\Bi(\nabla^gH)=\Bi(H^2)=0$, proving \ref{prop:generalized-Riemann-flat:BiH2:Riemannian} and \ref{prop:generalized-Riemann-flat:Rpm:Riemannian} in this case. Finally, Lemma \ref{lem:bismut-vs-riemann-curvature}\ref{lem:bismut-vs-riemann-curvature:Rpm-vs-Rg} gives
\[
0=R^\pm=R^g\pm\Alt_{12}(\nabla^gH)+\frac{1}{4}(H^2-\Bi(H^2))=R^g\pm\Alt_{12}(\nabla^gH)+\frac{1}{4}H^2.
\]
This implies that $\Alt_{12}(\nabla^gH)=0$, proving \ref{prop:generalized-Riemann-flat:Rg:Riemannian}. Together with $\Bi(\nabla^gH)=0$, it also implies that
\begin{align*}
\nabla^g_XH(Y,Z)&=\nabla^g_XH(Y,Z) +2\Alt_{12}(\nabla^gH)(Y,Z,X)\\
&=\nabla^g_XH(Y,Z)+\nabla^g_YH(Z,X)-\nabla^g_ZH(Y,X) \\
&=\Bi(\nabla^gH)(X,Y,Z)=0,
\end{align*}
finally giving \ref{prop:generalized-Riemann-flat:H-parallel:Riemannian} in this case as well.

Lastly, we turn to the results on sectional and scalar curvature. Since $\Bi(H^2)\in\Omega^4(M)$, we have that
\[
\prodesc{R^g(X,Y)Y}{X} = -\frac{1}{4}H^2(X,Y,Y,X)=\frac{1}{4}\norm{H(X,Y)}_g^2.
\]
This gives \ref{prop:generalized-Riemann-flat:sec:Riemannian}.

Also, if $\{E_j\}_j$ is a local orthonormal frame for $g$, then
\[
\scal^g=\sum_{j,k}\sec^g(E_j,E_k)=\frac{1}{4}\sum_{j,k}\norm{H(E_j,E_k)}_g^2=\frac{1}{4}\norm{H}_g^2.
\]
Since $\nabla^gH=0$, we also have that $d\norm{H}_g^2=0$, finally proving \ref{prop:generalized-Riemann-flat:scal:Riemannian}.
\end{proof}

Some easy consequences readily follow.

\begin{corollary} \label{cor:locally-symmetric-F-zero}
Let $g$ be Riemannian. If $\Tbb M\oplus\gfrak$ admits a flat generalized Bismut connection, then $(M,g)$ is locally symmetric.
\end{corollary}

\begin{proof}
By Proposition \ref{prop:generalized-Riemann-flat}\ref{prop:generalized-Riemann-flat:Rg:Riemannian} we have that $R^g$ depends only on $\Bi(H^2)$ and $H^2$. A straightforward computation gives that
\[
\nabla^g_XH^2(Y,Z,V,W)=g(\nabla^g_XH(Y,Z), H(V,W))+g(H(Y,Z),\nabla^g_XH(V,W))
\]
and that $\nabla^g_X(\Bi(H^2))=\Bi(\nabla^g_XH^2)$. Since $\nabla^gH=0$ by Proposition \ref{prop:generalized-Riemann-flat}\ref{prop:generalized-Riemann-flat:H-parallel:Riemannian}, we conclude that $\nabla^gR^g=0$, proving that $(M,g)$ is locally symmetric.
\end{proof}

\begin{corollary} \label{cor:trivial-gfrak}
Let $g$ be Riemannian and let $M$ be 1-connected. If $\Tbb M\oplus\gfrak$ admits a flat generalized Bismut connection $D$, then $\gfrak$ is isomorphic to a trivial Lie algebra bundle $M\times \gfrak$ (where now $\gfrak$ is a fixed Lie algebra) and, via this isomorphism, the connection $D$ restricted to $\gfrak$ is given by
\begin{equation} \label{eq:gfrak-part}
D_{X^\pm}r=\Lcal_Xr,\qquad D_r X^\pm=0, \qquad D_rs=\mu[r,s].
\end{equation}
On the other hand, the restriction of $D$ to $\Tbb M$ gives a flat generalized Bismut connection on $\Tbb M$.
\end{corollary}

\begin{proof}
By Proposition \ref{prop:generalized-Riemann-flat}\ref{prop:generalized-Riemann-flat:F-zero} we have that $\nabla$ is flat. If $M$ is 1-connected, then we can trivialize the bundle $\gfrak$ via parallel transport, in which case $\nabla$ becomes just the differential on $\Cinf(M,\gfrak)$, where now $\gfrak$ is a fixed Lie algebra. The expressions \eqref{eq:gfrak-part} now follow from Proposition \ref{prop:connection-tensors}, using that $F=0$. Because $F=0$ and $A=0$, it follows from Propositions \ref{prop:connection-tensors} and \ref{prop:curvatures-generalized-Bismut} that $D$ restricted to $\Tbb M$ defines a flat generalized Bismut connection on $\Tbb M$.
\end{proof}

Because of Corollary \ref{cor:trivial-gfrak}, in the rest of this section we will focus on flat generalized Bismut connections on $\Tbb M$ in the case that $(M,g)$ is Riemannian, since any flat generalized Bismut connection on $\Tbb M\oplus\gfrak$ is a trivial extension of such a generalized connection on $\Tbb M$ via \eqref{eq:gfrak-part}.

We will answer the following question: Which Riemannian manifolds $(M,g)$ can appear as the base of an exact Courant algebroid admitting a flat generalized Bismut connection?

To study this classification problem, we would like to restrict to ``irreducible'' objects, but at first it is not clear what the correct notion of irreducibility should be. For our purposes, the classical notion of $\nabla^g$-irreducibility will suffice. 

\begin{proposition}
If $(M,g)=(M_1,g_1)\times (M_2,g_2)$ is a Riemannian product such that $\Tbb M$ admits a flat generalized Bismut connection, then $H=H_1+H_2$ with $H_i\in\Omega^3(M_i)$.
\end{proposition}

\begin{proof}
The Riemannian product condition implies that $\nabla^g=\nabla^{g_1}\oplus\nabla^{g_2}$ and $R^g=R^{g_1}+R^{g_2}$. This gives that if $X_i\in\Xfrak(M_i)$, then $R^g(X_1,X_2,X_2,X_1)=0$, which means that $\sec^g(X_1,X_2)=0$. But then Proposition \ref{prop:generalized-Riemann-flat}\ref{prop:generalized-Riemann-flat:sec:Riemannian}
implies that $H(X_1,X_2)=0$.
\end{proof}

Hence, if $(M,g)$ is 1-connected and complete, then the triple $(M,g,H)$ decomposes as a product of irreducible triples. We restrict our problem, thus, to 1-connected, complete and irreducible (meaning $\nabla^g$-irreducible) spaces.

First of all, the case $\lambda^\pm\neq\mp\frac{1}{3}$ (the non-Levi-Civita case) reduces to the case of irreducible Bismut-flat manifolds, which was already classified by Cartan and Schouten \cite{cartan-schouten,agricola-friedrich-bismut-flat}.

\begin{proposition} \label{prop:flat-canonical-non-Levi-Civita}
Let $(M,g,H)$ be a triple such that $\Tbb M$ admits a flat generalized Bismut connection with $\lambda^\pm\neq\mp\frac{1}{3}$ and such that $H\neq 0$ and $(M,g)$ is a 1-connected, complete and irreducible Riemannian manifold. Then $(M,g)$ is a compact simple Lie group with a bi-invariant metric and $H$ is a multiple of the Cartan 3-form.
\end{proposition}

\begin{proof}
The triple $(M,g,H)$ is Bismut-flat, by Proposition \ref{prop:generalized-Riemann-flat}. By the Cartan--Schouten theorem on the classification of Bismut-flat manifolds, there are only two possible cases \cite[Thm. 2.2]{agricola-friedrich-bismut-flat}: either $M$ is a compact simple Lie group with a bi-invariant metric and $H$ a multiple of the Cartan 3-form, or $M$ is isometric to $\Sbb^7$. This depends on whether $\Bi(H^2)=0$ or $\Bi(H^2)\neq 0$, respectively. Since $\Bi(H^2)=0$, because $\lambda^\pm\neq \mp\frac{1}{3}$ (Proposition \ref{prop:generalized-Riemann-flat}\ref{prop:generalized-Riemann-flat:BiH2:Riemannian}), the second case is ruled out.
\end{proof}

We now turn to the case $\lambda^\pm=\mp\frac{1}{3}$, where the previous argument breaks down, because the manifold is a priori not Bismut-flat. Notice that in the case that $\mu=\frac{1}{3}$, this will be the Levi-Civita case.

Before turning to this case, though, let us recall some standard facts, which we will need, and which can be found for instance in \cite[Chap. V, Sect. 12]{bredon}. Let $G$ be a compact connected Lie group, with Lie algebra $\gfrak$, acting smoothly on a manifold $N$ and denote by $\Omega^\bullet(N)^G$ the cochain complex of left-invariant differential forms on $N$. Then averaging against a Haar measure on $G$ gives an isomorphism $H^\bullet(N,\R)\cong H^\bullet(\Omega^\bullet(N)^G)$. 

This allows us to write $H^\bullet(G,\R)$ in two different ways. First, if we let $G$ act on itself by left-translation, then $\Omega^\bullet(G)^G\cong \wedge^\bullet\gfrak^*$, by evaluating at the identity element of $G$, and hence $H^\bullet(G,\R)\cong H^\bullet(\gfrak,\R)$, the Chevalley--Eilenberg cohomology. Secondly, if we let $G\times G$ act on $G$ by $(g,h)\cdot k:=gkh^{-1}$, then $G\times G$-invariant forms are left-invariant forms which are also conjugation-invariant, also called bi-invariant. Hence, bi-invariant forms are identified with $(\wedge^\bullet\gfrak^*)^G$, with respect to the adjoint action of $G$ on $\gfrak$. Moreover, any $\Ad(G)$-invariant form on $\gfrak$ is closed, and since $H^\bullet(G,\R)\cong H^\bullet((\wedge^\bullet\gfrak^*)^G)$, we obtain finally that $H^\bullet(G,\R)\cong (\wedge^\bullet\gfrak^*)^G$, that is, each cohomology class of $G$ has a unique bi-invariant representative.

Some other facts that we will need are that $[\gfrak,\gfrak]=\gfrak$ if and only if $H^1(\gfrak,\R)=0$, and that this implies  $H^2(\gfrak,\R)=0$ as well. A slightly more elaborate argument gives also that if $H^1(\gfrak,\R)=0$ then the map $(S^2\gfrak^*)^G\to (\wedge^3\gfrak^*)^G$ sending $\kappa$ to the 3-form given by $H_\kappa(X,Y,Z):=\kappa([X,Y],Z)$ is an isomorphism. Lastly, if $\gfrak$ is simple, then $H^3(\gfrak,\R)=\R$, generated by the Cartan 3-form $H_\kappa$, where $\kappa$ is now the Killing form of $\gfrak$.

\begin{proposition} \label{prop:flat-canonical-Levi-Civita}
Let $(M,g,H)$ be a triple such that $\Tbb M$ admits a flat generalized Bismut connection with $\lambda^\pm=\mp\frac{1}{3}$ and such that $H\neq 0$ and $(M,g)$ is a 1-connected, complete and irreducible Riemannian manifold. Then $(M,g)$ is a compact simple Lie group with a bi-invariant metric and $H$ is a multiple of the Cartan 3-form.
\end{proposition}

\begin{proof}
By Proposition \ref{cor:locally-symmetric-F-zero}, $(M,g)$ is an irreducible symmetric space of compact type. Write $M=G/K$ for an irreducible symmetric pair $(G,K)$. Then there are two possibilities \cite[Prop. 6.40]{ziller}: either $G$ is compact simple, or $K$ is compact simple and $G=K\times K$, in which case $M=(K\times K)/K$ is isometric to $K$ with a bi-invariant metric. Write $\gfrak=\kfrak\oplus\pfrak$ for the corresponding Cartan decomposition, where we identify $\pfrak\cong T_xM$ for some fixed $x\in M$.

Because $M$ is an irreducible symmetric space, the holonomy representation at $x$ is equivalent to the adjoint representation of $K$ on $\pfrak$. Since $\nabla^gH=0$, the value of $H$ at $x$ gives an $\Ad(K)$-invariant 3-form on $\pfrak$, i.e., an element of $(\wedge^3\pfrak^*)^K$, which we also denote by $H$.

If we are in the second case of the above dichotomy, that is, if $K$ is compact simple and $G=K\times K$, then $M\cong K$ with a bi-invariant metric. The Cartan decomposition of $G$ is just $\gfrak=\kfrak\oplus\kfrak$, so by the comments in the previous paragraph, $H\in(\wedge^3\kfrak^*)^K$. Since this space is isomorphic to $H^3(\kfrak,\R)=\R$, we conclude that $H$ must be a multiple of the Cartan 3-form of $\kfrak$.

To conclude the proof, we will see that the first case of the dichotomy cannot happen unless $H=0$. Assume that $G$ is compact simple. Extend $H$ to a 3-form $\tilde H\in (\wedge^3\gfrak^*)^K$ by setting it to 0 on $\kfrak$. If $d_\gfrak$ denotes the differential of the Chevalley--Eilenberg complex, then we have that $d_\gfrak\tilde H=0$. Indeed, if we feed $d_\gfrak\tilde H$ with none, two, three or four elements in $\kfrak$, then it gives 0, because $i_\kfrak \tilde H=0$ and because of the Cartan relations $[\kfrak,\kfrak],[\pfrak,\pfrak]\subseteq\kfrak$. If $X_0\in\kfrak$ and the rest are in $\pfrak$, then
\begin{align*}
d_\gfrak\tilde H(X_0,X_1,X_2,X_3) &=-\tilde H([X_0,X_1],X_2,X_3)+\tilde H([X_0,X_2],X_1,X_3)\\
&\qquad\qquad -\tilde H([X_0,X_3],X_1,X_2)=0,
\end{align*}
because $\tilde H$ is $\Ad(K)$-invariant.

Since $H^3(\gfrak,\R)\cong (\wedge^3\gfrak^*)^G\cong (S^2\gfrak^*)^G$, there is a unique $\tilde \kappa\in(S^2\gfrak^*)^G$ such that $\tilde H=H_{\tilde \kappa}+d_\gfrak\alpha$, for some $\alpha\in \wedge^2\gfrak^*$. Since both $\tilde H$ and $H_{\tilde \kappa}$ are $\Ad(K)$-invariant and $K$ is compact, by averaging over $K$ we can actually take $\alpha\in(\wedge^2\gfrak^*)^K$. In this case, if $X\in\kfrak$, then
\[
\tilde \kappa(X,[\cdot,\cdot])=i_X H_{\tilde \kappa}=-i_X d_\gfrak\alpha=-\alpha(X,[\cdot,\cdot]).
\]
Since $[\gfrak,\gfrak]=\gfrak$, we conclude that $\tilde \kappa(Y,X)=\alpha(Y,X)$ for all $Y\in\gfrak$. In particular, $\tilde \kappa(X,X)=0$, which means that $\tilde \kappa|_{\kfrak\times\kfrak}$ is both symmetric and skew-symmetric, i.e., it vanishes. Since $\tilde \kappa$ is an invariant symmetric bilinear form and $\gfrak$ is simple, it must be a multiple of the Killing form of $\gfrak$. But since the Killing form is non-degenerate (as $\gfrak$ is compact) and $\kfrak$ and $\pfrak$ are orthogonal to each other (by the Cartan relations), then $\tilde \kappa|_{\kfrak\times\kfrak}=0$ if and only if this multiple is 0, i.e., if and only if $\tilde \kappa=0$. Then we have $\tilde H=d_\gfrak\alpha$ and  $i_\kfrak\alpha=0$, which imply that $H=0$, because of the Cartan relation $[\pfrak,\pfrak]\subseteq\kfrak$.
\end{proof}

From the previous discussion, we conclude our main result.

\begin{theorem} \label{thm:flat-canonical-all}
Let $(M,g,H,\gfrak,\nabla,F)$ be a tuple with $H\neq 0$ and $(M,g)$ a 1-connected, complete and irreducible Riemannian manifold, and such that $\Tbb M\oplus\gfrak$ admits a flat generalized Bismut connection $D$. Then
\begin{enumerate}[label=\upshape{(\roman*)}]
    \item $F=0$ and $\gfrak$ is trivializable in such a way that $D$ becomes a trivial extension of a flat generalized Bismut connection on $\Tbb M$ via \eqref{eq:gfrak-part},
    \item $(M,g,H)$ is a compact simple Lie group with a bi-invariant metric and a multiple of the Cartan 3-form,
\end{enumerate}
\end{theorem}

\bigskip
\section{On Lie groups} \label{sect:lie-groups}

In the previous section we saw that the building blocks of spaces with flat canonical generalized Levi-Civita connections are 1-connected compact simple Lie groups with a bi-invariant metric and a multiple of the Cartan 3-form. A natural question to ask, now, is: 
Are there non-flat generalized Levi-Civita connections on these Lie groups?

The answer is yes, and many, as we will see. As is common with Lie groups, we will focus on left-invariant connections. We focus on generalized connections on $\Tbb M$, since these can be trivially extended to $\Tbb M\oplus\gfrak$ via \eqref{eq:gfrak-part}.

Let $G$ be a Lie group with Lie algebra $\gfrak$ and a bi-invariant metric $g$, and let $H$ be the Cartan 3-form. Consider the left action of $G$ on itself, which we denote by $L$. It lifts to actions on $TG$ and $T^*G$:
\[
\begin{tikzcd}
	TG & TG \\
	G & G
	\arrow["{L_{h*}}", from=1-1, to=1-2]
	\arrow[from=1-1, to=2-1]
	\arrow["{L_h}"', from=2-1, to=2-2]
	\arrow[from=1-2, to=2-2]
\end{tikzcd} \qquad \text{and} \qquad
\begin{tikzcd}
	{T^*G} & {T^*G} \\
	G & G
	\arrow["{L_{h^{-1}}^*}", from=1-1, to=1-2]
	\arrow[from=1-1, to=2-1]
	\arrow["{L_h}"', from=2-1, to=2-2]
	\arrow[from=1-2, to=2-2]
\end{tikzcd},\qquad\text{for $h\in G$.}
\]
Let us denote by $L$ also the $G$-action on $\Tbb G$, i.e.,
\[
L_h(X+r)=L_{h*}X+L_{h^{-1}}^*r,\quad\text{for $h\in G$ and $X+r\in\Tbb G$.}
\]
It is well-known that the diffeomorphisms preserving the Courant algebroid structure on $\Tbb G$ are those which leave $H$ invariant. Since $H$ is left-invariant, this means that the action $L$ on $\Tbb G$ is by Courant algebroid automorphisms.

We denote by $L$ as well the induced action on sections of $\Tbb G$, i.e., if $a\in\Gamma(\Tbb G)$ and $h\in G$, then
\[
L_ha(k):=L_h(a(h^{-1}k)),\quad\text{for $k\in G$.}
\]

\begin{definition}
A generalized connection $D$ on $\Tbb G$ is said to be \textbf{left-invariant} if
\[
D_{L_ha}(L_hb)=L_h(D_ab),\quad\text{for all $a,b\in\Gamma(\Tbb G)$ and $h\in G$.}
\]
\end{definition}

Recall that by Proposition \ref{prop:connection-tensors} a metric generalized connection with pure-type torsion on $\Tbb G$ is given by tensors $B^\pm\in\Omega^1(G,\so(TG))$, and that the connection is Levi-Civita if and only if $\Bi(B^\pm)=\mp H$.

\begin{lemma} \label{lem:G-equivariant-Levi-Civita}
A metric generalized connection with pure-type torsion on $\Tbb G$ is left-invariant if and only if both $B^\pm$ are left-invariant.
\end{lemma}

\begin{proof}
It is straightforward to see that if $X\in TG$ and $h\in G$, then $(L_{h*}X)^\pm=L_h(X^\pm)$. Recall that for any isometry $\varphi$ of $G$ and $X,Y\in\Xfrak(G)$ we have that $\nabla^g_{\varphi_*X}(\varphi_*Y)=\varphi_*(\nabla^g_XY)$. Since $L_h$ is an isometry and $H$ is left-invariant, we have that
\[
\nabla^\pm_{L_{h*}X}(L_{h*}Y)=L_{h*}(\nabla^\pm_XY),
\]
which gives that
\[
D_{L_hX^\mp}(L_hY^\pm)=D_{(L_{h*}X)^\mp}(L_{h*}Y)^\pm=(L_{h*}(\nabla^\pm_XY))^\pm=L_h(D_{X^\mp}Y^\pm).
\]
Similarly, writing $\nablabarpm:=\nabla^\pm+B^\pm$,
\[
\nablabarpm_{L_{h*}X}(L_{h*}Y)=L_{h*}(\nabla^\pm_XY)+L_{h*}(L_h^*B^\pm(X,Y))=L_{h*}(\nabla^\pm_XY+L_h^*B^\pm(X,Y)),
\]
whereas
\[
L_h(D_{X^\pm}Y^\pm)=(L_{h*}(\nabla^\pm_XY+B^\pm(X,Y)))^\pm.
\]
It is clear now that $D$ is left-invariant if and only if $L_h^*B^\pm=B^\pm$ for all $h\in G$.
\end{proof}

For the classification of flat left-invariant generalized Levi-Civita connections, consider the following sequence:
\begin{equation}\label{eq:es-alternation}
\begin{tikzcd}
	0 & {S^3\gfrak^*} & {S^2\gfrak^*\otimes \gfrak^*} & {\gfrak^*\otimes\wedge^2\gfrak^*} & {\wedge^3\gfrak^*} & 0
	\arrow[from=1-1, to=1-2]
	\arrow[from=1-2, to=1-3]
	\arrow["{\Alt_{23}}", from=1-3, to=1-4]
	\arrow["\Bi", from=1-4, to=1-5]
	\arrow[from=1-5, to=1-6]
\end{tikzcd}
\end{equation}
where $\Alt_{23}$ is the skew-symmetrization of the second and third components.

The sequence \eqref{eq:es-alternation} is just a piece of the Koszul complex of the Weil algebra of $\gfrak$, and since the Weil algebra is acyclic, it is actually exact. We give a direct proof of this fact.

\begin{lemma} \label{lem:weil-algebra-acyclic}
The sequence \eqref{eq:es-alternation} is exact. Moreover, the sequence remains exact when restricted to $\Ad(G)$-invariant tensors.
\end{lemma}

\begin{proof}
Exactness at the first, second and fourth nodes is a simple check, since $\Bi$ on elements of $\gfrak^*\otimes\wedge^2\gfrak^*$ is just the total skew-symmetrization. Let us check exactness at the third node. Let $B\in\gfrak^*\otimes\wedge^2\gfrak^*$ such that $\Bi(B)=0$, and define $B'\in S^2\gfrak^*\otimes\gfrak^*$ by
\[
B'(X,Y,Z):=\frac{2}{3}(B(X,Y,Z)+B(Y,X,Z)),\quad\text{ for $X,Y,Z\in\gfrak$.}
\]
Then
\begin{align*}
(\Alt_{23}B')(X,Y,Z) &=\frac{1}{3}(B(X,Y,Z)+B(Y,X,Z)-B(X,Z,Y)-B(Z,X,Y))\\
&=B(X,Y,Z).
\end{align*}
Moreover, it is straightforward to check that $\Alt_{23}$ restricted to $(S^2\gfrak^*\otimes\gfrak^*)^G$ maps to $(\gfrak^*\otimes\wedge^2\gfrak^*)^G$ and that $\Bi$ restricted to $(\gfrak^*\otimes\wedge^2\gfrak^*)^G$ maps to $(\wedge^3\gfrak^*)^G$, and the same reasoning as before gives exactness also in this case.
\end{proof}

\begin{theorem} \label{thm:lie-groups}
The space of flat left-invariant generalized Levi-Civita connections on $\Tbb G$ is an affine space modeled on
\[
(S^2\gfrak^*\otimes\gfrak^*)^G/(S^3\gfrak^*)^G \oplus (S^2\gfrak^*\otimes\gfrak^*)/S^3\gfrak^*.
\]
In particular, there are non-flat left-invariant generalized Levi-Civita connections on $\Tbb G$.
\end{theorem}

\begin{proof}
Let $D$ be a left-invariant generalized Levi-Civita connection on $\Tbb G$. By Lemma \ref{lem:G-equivariant-Levi-Civita}, it is defined by a tensor $B^\pm\in\gfrak^*\otimes\wedge^2\gfrak^*$ such that $\Bi(B^\pm)=\mp H$ (regarding $H$ as an element of $\wedge^3\gfrak^*$, since it is left-invariant). By Proposition \ref{prop:curvatures-regular-CA}, we have that $D$ is flat if and only if $\nabla^\pm B^\pm=0$, since $G$ is Bismut-flat. Recall that the Levi-Civita connection of $G$ on left-invariant vector fields $X,Y$ is given by $\nabla^g_XY=\frac{1}{2}[X,Y]$. Hence, $\nabla^\pm_XY=\frac{1\pm 1}{2}[X,Y]$, from which we obtain, for $X,Y,Z,W\in\gfrak$,
\[
\nabla^\pm_XB^\pm(Y,Z,W)=-\frac{1\pm 1}{2}(B^\pm([X,Y],Z,W)+B^\pm(Y,[X,Z],W)+B^\pm(Y,Z,[X,W])).
\]
From this we obtain that $\nabla^-B^-=0$ and that $\nabla^+B^+=0$ if and only if $B^+$ is $\Ad(G)$-invariant.

Let now $D^i$, for $i=1,2$, be flat left-invariant generalized Levi-Civita connections on $\Tbb G$, defined by tensors $B^\pm_i\in\gfrak^*\otimes\wedge^2\gfrak^*$ such that $\Bi(B^\pm_i)=\mp H$. Then, by Lemma \ref{lem:weil-algebra-acyclic},
\[
B^+_2-B^+_1\in \ker\Bi\cap (\gfrak^*\otimes\wedge^2\gfrak^*)^G=(S^2\gfrak^*\otimes\gfrak^*)^G/(S^3\gfrak^*)^G
\]
and
\[
B^-_2-B^-_1\in\ker\Bi\cap(\gfrak^*\otimes \wedge^2\gfrak^*)=(S^2\gfrak^*\otimes\gfrak^*)/S^3\gfrak^*.
\]
Since the space of all left-invariant generalized Levi-Civita connections on $\Tbb G$ is an affine space \cite{garcia-fernandez-spinors} modeled on
\[
(S^2\gfrak^*\otimes\gfrak^*)/S^3\gfrak^* \oplus (S^2\gfrak^*\otimes\gfrak^*)/S^3\gfrak^*,
\]
it easily follows that there are many non-flat generalized Levi-Civita connections on $G$, corresponding to all tensors in $(S^2\gfrak^*\otimes\gfrak^*)/S^3\gfrak^*$ which are not $\Ad(G)$-invariant.
\end{proof}

\bigskip
\appendix
\section{The pseudo-Riemannian case} \label{app:pseudo-riemannian}

We include here the analogue of Proposition \ref{prop:generalized-Riemann-flat} when $g$ is not necessarily Riemannian. The result and the proof are considerably more convoluted, but it poses an interesting question, which we include at the end of this section.

\begin{proposition} \label{prop:generalized-Riemann-flat:app}
If $\Tbb M\oplus\gfrak$ admits a flat generalized Bismut connection, then
\begin{enumerate}[label=\upshape{(\roman*)}]
    \item $\mu F_{[r,s]}=-\frac{1}{2}F_r F_s$, \label{prop:generalized-Riemann-flat:F-bracket}
    \item $F$ is null, by which we mean that $\norm{F_r X}_g^2=0$, \label{prop:generalized-Riemann-flat:F-null}
	\item $H(F_r X,Y,Z)=H(X,F_r Y,Z)=H(X,Y,F_r Z)$. \label{prop:generalized-Riemann-flat:F-Horthogonal}
\end{enumerate}
Also, the following properties:
\begin{enumerate}[resume,label=\upshape{(\roman*)}]
	\item $F$ is isotropic, by which we mean that $F^2=0$, \label{prop:generalized-Riemann-flat:F-isotropic}
    \item $\nabla^gH=0$, \label{prop:generalized-Riemann-flat:H-parallel}
    \item $R^\pm=-\frac{1}{6}\Bi(H^2)$, \label{prop:generalized-Riemann-flat:Rpm}
	\item $R^g=\frac{1}{12}\Bi(H^2)-\frac{1}{4}H^2$, \label{prop:generalized-Riemann-flat:Rg}
	\item $\sec^g(X,Y)=\displaystyle \frac{1}{4}\frac{\norm{H(X,Y)}_g^2}{\norm{X\wedge Y}_g^2}$, \label{prop:generalized-Riemann-flat:sec}
	\item $\scal^g=\frac{1}{4}\norm{H}_g^2$, which is locally constant, \label{prop:generalized-Riemann-flat:scal}
\end{enumerate}
hold in the case that one of the following conditions is satisfied:
\begin{enumerate}[label=\upshape{(\alph*)}]
    \item $(1+2\lambda^+)(1-3\lambda^-)\neq 0$, \label{prop:generalized-Riemann-flat:degenerate-lambdas}
    \item $\lambda^\pm=\mp\frac{1}{3}$. \label{prop:generalized-Riemann-flat:levi-civita}
\end{enumerate}
Furthermore,
\begin{enumerate}[start=10,label=\upshape{(\roman*)}]
    \item if $\mu\neq -1$, then $\nabla$ is flat, \label{prop:generalized-Riemann-flat:nabla-flat}
	\item if $(\lambda^+,\lambda^-)\neq (-\frac{5}{9},-\frac{2}{3})$, then $H(F_r X,Y)=0$ and $\nabla^gF=0$, \label{prop:generalized-Riemann-flat:HF-zero}
	\item if $(1+2\lambda^+)(1-3\lambda^-)\neq 0$, then $\Bi(H^2)=0$. \label{prop:generalized-Riemann-flat:BiH2}
\end{enumerate}
\end{proposition}

\begin{proof}
First of all, \ref{prop:generalized-Riemann-flat:F-bracket} and \ref{prop:generalized-Riemann-flat:F-null} are proven exactly in the same manner as for the Riemannian case, see the proof of Proposition \ref{prop:generalized-Riemann-flat}. Notice that now \ref{prop:generalized-Riemann-flat:F-null} does not imply $F=0$, since $g$ is not necessarily Riemannian.

Secondly, from \eqref{eq:curvaturem-XYxi} we have that $(1+\mu)[F(X,Y),r]=0$. If $\mu\neq -1$, then $\nabla$ is flat, establishing \ref{prop:generalized-Riemann-flat:nabla-flat}.

Going further, \eqref{eq:curvaturem-XYxi} gives
\begin{equation} \label{eq:nablamF}
(\nabla^-_XF)_r Y=F_r (H(X,Y))-\lambda^- H(X, F_r Y),
\end{equation}
and \eqref{eq:curvaturem-xiYZ},
\[
(\nabla^-_XF)_r Y=\frac{\lambda^-}{2}H(F_r X,Y).
\]
from which we conclude that
\begin{equation} \label{eq:FxiHXY}
F_r (H(X,Y))=\frac{\lambda^-}{2}(H(F_r X,Y)+2H(X,F_r Y)).
\end{equation}
Since the left-hand side is skew-symmetric in $X$ and $Y$, the right-hand side must be so as well. If $\lambda^-=0$, then $F_r(H(X,Y))=0$, which would also mean that $H(F_rX,Y)=H(X,F_rY)=0$, proving \ref{prop:generalized-Riemann-flat:F-Horthogonal} and \ref{prop:generalized-Riemann-flat:HF-zero} in this case. If $\lambda^-\neq 0$, then
\[
H(F_r X,Y)+2H(X,F_r Y)=-H(F_r Y, X)-2H(Y,F_r X),
\]
and this is equivalent to $H(F_r X,Y)=H(X,F_r Y)$. This implies in turn that $H(F_r X,Y)=-F_r (H(X,Y))$, proving \ref{prop:generalized-Riemann-flat:F-Horthogonal} in general. We can, therefore, rewrite \eqref{eq:FxiHXY} as
\[
\left( 1+\frac{3}{2}\lambda^-\right) H(F_r X,Y)=0.
\]
If $\lambda^-\neq -\frac{2}{3}$, then $H(F_r X,Y)=0$. On the other hand, \eqref{eq:curvaturep-XxiZ} gives
\[
(\nabla^+_XF)_r Y=-\lambda^+(H(X,F_r Y)-F_r (H(X,Y))+H(F_r X, Y)),
\]
and by \ref{prop:generalized-Riemann-flat:F-Horthogonal} this equals $-3\lambda^+H(F_r X,Y)$. We also have in general that
\[
(\nabla^+_XF)_r Y=(\nabla^-_XF)_r Y +H(X,F_r Y)-F_r (H(X,Y)),
\]
and by \eqref{eq:nablamF} this equals $(1-\lambda^-)H(F_r X, Y)$.
Hence, we have that
\[
(1-\lambda^-+3\lambda^+)H(F_r X,Y)=0.
\]
If $\lambda^-=-\frac{2}{3}$, this equation becomes
\[
(5+9\lambda^+)H(F_r X,Y)=0.
\]
The conclusion of this discussion is that if $(\lambda^+,\lambda^-)\neq (-\frac{5}{9},-\frac{2}{3})$, then $H(F_r X,Y)=0$. In this case, we have that $\nabla^+F=\nabla^-F=0$ and, thus, also $\nabla^gF=0$, establishing finally \ref{prop:generalized-Riemann-flat:HF-zero}.

We now focus on \eqref{eq:curvaturep-XYZ} and \eqref{eq:curvaturem-XYZ}. We have that
\begin{equation}\label{eq:nablapm-Rpm}
\lambda^\pm\nabla^\pm_XH(Y,Z)=-R^\pm(X,Y)Z+\epsilon_\pm F_{F(X,Y)}Z,
\end{equation}
where $\epsilon_+=\frac{1}{2}$ and $\epsilon_-=-1$. We can rewrite this as
\begin{equation}\label{eq:nablapm-Rpm-noXYZ}
\lambda^\pm\nabla^\pm H=-R^\pm+\epsilon_\pm F^2.
\end{equation}
The first easy consequence of \eqref{eq:nablapm-Rpm} is that, since the right-hand side is skew-symmetric in $X$ and $Y$, then the left-hand side must be so as well. If $\lambda^+\neq 0$ or $\lambda^-\neq 0$, then this means that either $\nabla^+H$ or $\nabla^-H$ actually lies in $\Omega^4(M)$, and Lemma \ref{lem:bismut-vs-riemann-curvature}\ref{lem:bismut-vs-riemann-curvature:nablapmH} then implies that $\nabla^gH\in\Omega^4(M)$ as well, since $\Bi(H^2)\in\Omega^4(M)$. But then 
\begin{equation}\label{eq:nablagH}
\nabla^gH=\frac{1}{4}dH=\frac{1}{2}\Bi(F^2),    
\end{equation}
so we conclude from \eqref{eq:nablapm-Rpm-noXYZ} and Lemma \ref{lem:bismut-vs-riemann-curvature}\ref{lem:bismut-vs-riemann-curvature:nablapmH} that
\begin{align}
R^\pm &= -\lambda^\pm\nabla^\pm H+\epsilon_\pm F^2 \nonumber \\
&=-\lambda^\pm(\nabla^gH\mp\frac{1}{2}\Bi(H^2))+\epsilon_\pm F^2 \nonumber \\
&=\pm\frac{\lambda^\pm}{2}\Bi(H^2)-\frac{\lambda^\pm}{2} \Bi(F^2) +\epsilon_\pm F^2. \label{eq:Rpm}
\end{align}
These curvatures $R^\pm$ have to satisfy Lemma \ref{lem:bismut-vs-riemann-curvature}\ref{lem:bismut-vs-riemann-curvature:RBismut-sym}. Since $\Bi(H^2)$, $\Bi(F^2)$ and $F^2$ are sections of $S^2\wedge^2T^*M$, this gives, from \eqref{eq:Rpm},
\[
\lambda^+\Bi(H^2)-\lambda^+\Bi(F^2)+F^2=-\lambda^-\Bi(H^2)-\lambda^-\Bi(F^2)-2F^2+2\Bi(F^2).
\]
Thus,
\[
3F^2=-(\lambda^++\lambda^-)\Bi(H^2)+(\lambda^+-\lambda^-+2)\Bi(F^2).
\]
This implies that $F^2\in\Omega^4(M)$, so $\Bi(F^2)=3F^2$ and this equation becomes
\begin{equation} \label{eq:F2-BiH2}
3\left(1+\lambda^+-\lambda^-\right)F^2=(\lambda^++\lambda^-)\Bi(H^2).
\end{equation}
We will revisit this equation later.

The (first) Bianchi identity for $R^\pm$ in general states that
\begin{equation} \label{eq:bianchi-identity-Rpm}
\Bi(R^\pm)=\Bi(H^2)\pm\Bi(\nabla^\pm H).
\end{equation}
Using \eqref{eq:nablapm-Rpm-noXYZ} we now have that
\[
\lambda^\pm\Bi(R^\pm)=\lambda^\pm(\Bi(H^2)\pm\Bi(\nabla^\pm H))=\lambda^\pm\Bi(H^2)\mp\Bi(R^\pm)\pm\epsilon_\pm\Bi(F^2),
\]
i.e.,
\begin{equation} \label{eq:bianchi-Rpm}
(1\pm\lambda^\pm)\Bi(R^\pm)=\pm\lambda^\pm\Bi(H^2)+\epsilon_\pm\Bi(F^2).
\end{equation}
Since the left-hand side of \eqref{eq:nablapm-Rpm} is skew-symmetric on $Y$ and $Z$, then the right-hand side must be so as well, so
\[
R^\pm(X,Y)Z-\epsilon_\pm F_{F(X,Y)}Z=-R^\pm(X,Z)Y+\epsilon_\pm F_{F(X,Z)}Y.
\]
By skew-symmetry of the left-hand side of this equation on $X$ and $Y$, we also get
\[
R^\pm(X,Y)Z-\epsilon_\pm F_{F(X,Y)}Z=-R^\pm(Z,Y)X+\epsilon_\pm F_{F(Z,Y)}X.
\]
Hence,
\begin{align*}
\Bi(R^\pm)(X,Y,Z) &=R^\pm(X,Y)Z+R^\pm(Y,Z)X+R^\pm(Z,X)Y \\
&=3R^\pm(X,Y)Z+\epsilon_\pm(F_{F(Y,Z)}X+F_{F(Z,X)}Y-2F_{F(X,Y)}Z),
\end{align*}
which we can rewrite as
\[
\Bi(R^\pm)=3R^\pm+\epsilon_\pm(\Bi(F^2)-3F^2).
\]
Combined with \eqref{eq:bianchi-Rpm}, this gives that
\begin{align*}
3(1\pm\lambda^\pm)R^\pm+\epsilon_\pm(1\pm\lambda^\pm)(\Bi(F^2)-3F^2)=\pm\lambda^\pm\Bi(H^2)+\epsilon_\pm\Bi(F^2).
\end{align*}
Reordering the equation, we obtain
\begin{equation} \label{eq:intermmediate-1}
3(1\pm\lambda^\pm)R^\pm=\pm\lambda^\pm(\Bi(H^2)-\epsilon_\pm\Bi(F^2))+3\epsilon_\pm(1\pm\lambda^\pm)F^2.
\end{equation}

We can rewrite the left-hand side of this equation, using \eqref{eq:nablapm-Rpm-noXYZ} and Lemma \ref{lem:bismut-vs-riemann-curvature}\ref{lem:bismut-vs-riemann-curvature:nablapmH}, as
\begin{align}\label{eq:3(1+lambda)R}
3(1\pm\lambda^\pm)R^\pm &= -3(1\pm\lambda^\pm)\lambda^\pm\nabla^\pm H+3\epsilon_\pm(1\pm\lambda^\pm)F^2 \nonumber \\
&=-3(1\pm\lambda^\pm)\lambda^\pm\nabla^gH+3(1\pm\lambda^\pm)\left( \epsilon_\pm F^2\pm\frac{1}{2}\lambda^\pm\Bi(H^2)\right).
\end{align}
Using \eqref{eq:3(1+lambda)R} in \eqref{eq:intermmediate-1} and reordering the terms we are led to
\begin{equation} \label{eq:nablagH:app}
6\lambda^\pm (1\pm\lambda^\pm)\nabla^gH=\pm 2\lambda^\pm\epsilon_\pm\Bi(F^2)\pm\lambda^\pm(1\pm 3\lambda^\pm)\Bi(H^2).
\end{equation}
Using that $\nabla^gH=\frac{1}{2}\Bi(F^2)$ and reordering the terms we get
\begin{equation} \label{eq:BiH2-BiF2}
\lambda^\pm(3\mp2\epsilon_\pm\pm 3\lambda^\pm)\Bi(F^2)=\pm\lambda^\pm(1\pm 3\lambda^\pm)\Bi(H^2).
\end{equation}
Using that $F^2\in\Omega^4(M)$, these two equations separately look like
\begin{align}
3\lambda^+(2+3\lambda^+)F^2 &=\lambda^+(1+3\lambda^+)\Bi(H^2), \label{eq:F2-BiH2-2}\\
3\lambda^-(1-3\lambda^-)F^2 &=-\lambda^-(1-3\lambda^-)\Bi(H^2). \label{eq:F2-BiH2-3} 
\end{align}
Then $\lambda^-\eqref{eq:F2-BiH2-2}-\lambda^+\eqref{eq:F2-BiH2-3}$ gives
\[
3\lambda^+\lambda^-(1+3\lambda^++3\lambda^-)F^2=\lambda^+\lambda^-(2+3\lambda^+-3\lambda^-)\Bi(H^2).
\]
Multiplying this equation by $(\lambda^++\lambda^-)$ and using \eqref{eq:F2-BiH2} in the right-hand side, we obtain
\[
\lambda^+\lambda^-(1+3\lambda^++3\lambda^-)(\lambda^++\lambda^-)F^2=\lambda^+\lambda^-(2+3\lambda^+-3\lambda^-)(1+\lambda^+-\lambda^-)F^2.
\]
A patient calculation will show that this equation simplifies to
\[
\lambda^+\lambda^-(1+2\lambda^+)(1-3\lambda^-)F^2=0.
\]
Assume now that $\lambda^+\neq 0,-\frac{1}{2}$ and $\lambda^-\neq 0,\frac{1}{3}$. Then $F^2=0$, proving \ref{prop:generalized-Riemann-flat:F-isotropic} in this case. In particular, this gives that $\nabla^gH=\frac{1}{2}\Bi(F^2)=0$, proving \ref{prop:generalized-Riemann-flat:H-parallel}. It also gives, from \eqref{eq:Rpm},
\[
R^\pm=\pm\frac{\lambda^\pm}{2}\Bi(H^2).
\]

Moreover, \eqref{eq:BiH2-BiF2} now gives that $\lambda^-(1-3\lambda^-)\Bi(H^2)=0$, which implies that $\Bi(H^2)=0$, establishing \ref{prop:generalized-Riemann-flat:BiH2}. In this case, we have that $R^\pm=0=-\frac{1}{6}\Bi(H^2)$, settling \ref{prop:generalized-Riemann-flat:Rpm}. From Lemma \ref{lem:bismut-vs-riemann-curvature}\ref{lem:bismut-vs-riemann-curvature:Rpm-vs-Rg} it now follows that $R^g=\frac{1}{12}\Bi(H^2)-\frac{1}{4}H^2$, giving \ref{prop:generalized-Riemann-flat:Rg}.

Let us now study the cases where the previous argument breaks down, that is, when
\[
\lambda^+\lambda^-(1+2\lambda^+)(1-3\lambda^-)=0.
\]

First, if $\lambda^+\neq 0$ and $\lambda^-=0$, then \eqref{eq:F2-BiH2} and \eqref{eq:F2-BiH2-2} give
\begin{align}
3(1+\lambda^+)F^2 &=\lambda^+\Bi(H^2) \label{eq:F2-BiH2-lambda+}\\
(2+3\lambda^+)F^2 &=\left(\frac{1}{3}+\lambda^+\right)\Bi(H^2). \nonumber
\end{align}
Substracting these two gives
\[
F^2=-\frac{1}{3}\Bi(H^2).
\]
But then, \eqref{eq:F2-BiH2-lambda+} gives
\[
(1+2\lambda^+)F^2=0.
\]
Therefore, $F^2=0$ if $\lambda^+\neq -\frac{1}{2}$. In this case, then, items \ref{prop:generalized-Riemann-flat:F-isotropic}, \ref{prop:generalized-Riemann-flat:H-parallel}, \ref{prop:generalized-Riemann-flat:Rpm}, \ref{prop:generalized-Riemann-flat:Rg} and \ref{prop:generalized-Riemann-flat:BiH2} follow.

If now $\lambda^-\neq 0$ and $\lambda^+=0$, then \eqref{eq:F2-BiH2} and \eqref{eq:F2-BiH2-3} give
\begin{align}
3(1-\lambda^-)F^2 &=\lambda^-\Bi(H^2) \label{eq:F2-BiH2-lambda-}\\
(1-3\lambda^-)F^2 &=\left(-\frac{1}{3}+\lambda^-\right)\Bi(H^2). \nonumber
\end{align}
Substracting these two gives
\[
2F^2=\frac{1}{3}\Bi(H^2).
\]
But then, \eqref{eq:F2-BiH2-lambda-} gives
\[
(1-3\lambda^-)F^2=0.
\]
Therefore, $F^2=0$ if $\lambda^-\neq \frac{1}{3}$, and in this case also items \ref{prop:generalized-Riemann-flat:F-isotropic}, \ref{prop:generalized-Riemann-flat:H-parallel}, \ref{prop:generalized-Riemann-flat:Rpm}, \ref{prop:generalized-Riemann-flat:Rg} and \ref{prop:generalized-Riemann-flat:BiH2} follow.

Consider now the case $\lambda^\pm=0$. In this case, most of our arguments break down, since we do not even know whether $\nabla^gH$ is a 4-form. In this case, $R^\pm=\epsilon_\pm F^2$ directly holds, by \eqref{eq:nablapm-Rpm}. Lemma \ref{lem:bismut-vs-riemann-curvature}\ref{lem:bismut-vs-riemann-curvature:RBismut-sym} together with \eqref{eq:nablagH:app} implies now that
\[
\frac{1}{2}F^2=-F^2+\Bi(F^2),
\]
which means that $F^2=\frac{2}{3}\Bi(F^2)\in\Omega^4(M)$, in which case we have that $F^2=2F^2$, that is, $F^2=0$ again. Hence, $R^\pm=0$, and then the Bianchi identity \eqref{eq:bianchi-identity-Rpm} gives
\[
0=\Bi(H^2)\pm \Bi(\nabla^\pm H).
\]
Using Lemma \ref{lem:bismut-vs-riemann-curvature}\ref{lem:bismut-vs-riemann-curvature:nablapmH}, we obtain
\[
0=\Bi(H^2)\pm\Bi(\nabla^gH)-\frac{1}{2}\Bi(\Bi(H^2))=\pm\Bi(\nabla^gH)-\frac{1}{2}\Bi(H^2).
\]
This implies that $\Bi(\nabla^gH)=\Bi(H^2)=0$, proving \ref{prop:generalized-Riemann-flat:Rpm} and \ref{prop:generalized-Riemann-flat:BiH2} in this case. Finally, Lemma \ref{lem:bismut-vs-riemann-curvature}\ref{lem:bismut-vs-riemann-curvature:Rpm-vs-Rg} gives
\[
0=R^\pm=R^g\pm\Alt_{12}(\nabla^gH)+\frac{1}{4}(H^2-\Bi(H^2))=R^g\pm\Alt_{12}(\nabla^gH)+\frac{1}{4}H^2.
\]
This implies that $\Alt_{12}(\nabla^gH)=0$, proving \ref{prop:generalized-Riemann-flat:Rg}. Together with $\Bi(\nabla^gH)=0$, it also implies that
\begin{align*}
\nabla^g_XH(Y,Z)&=\nabla^g_XH(Y,Z) +2\Alt_{12}(\nabla^gH)(Y,Z,X)\\
&=\nabla^g_XH(Y,Z)+\nabla^g_YH(Z,X)-\nabla^g_ZH(Y,X) \\
&=\Bi(\nabla^gH)(X,Y,Z)=0,
\end{align*}
finally giving \ref{prop:generalized-Riemann-flat:H-parallel} in this case as well.

The conclusion of this discussion is that we have reduced the potentially problematic cases to those satisfying the equation
\[
(1+2\lambda^+)(1-3\lambda^-)=0,
\]
finally proving that the condition \ref{prop:generalized-Riemann-flat:degenerate-lambdas} suffices.

As for condition \ref{prop:generalized-Riemann-flat:levi-civita}, if $\lambda^-=\frac{1}{3}$, then \eqref{eq:F2-BiH2} gives
\[
3(2+3\lambda^+)F^2=(1+3\lambda^+)\Bi(H^2).
\]
If $\lambda^+=-\frac{1}{3}$, then we obtain $F^2=0$ again. Hence, $R^\pm=-\frac{1}{6}\Bi(H^2)$ by \eqref{eq:Rpm}. Then the Bianchi identity \eqref{eq:bianchi-identity-Rpm} and Lemma \ref{lem:bismut-vs-riemann-curvature}\ref{lem:bismut-vs-riemann-curvature:nablapmH} give
\[
-\frac{1}{2}\Bi(H^2)=\Bi(R^\pm)=\Bi(H^2)\pm\Bi(\nabla^\pm H)=\pm\Bi(\nabla^gH)-\frac{1}{2}\Bi(H^2).
\]
This implies that $\Bi(\nabla^gH)=0$. A similar argument as we did for $\lambda^\pm=0$, using Lemma \ref{lem:bismut-vs-riemann-curvature}\ref{lem:bismut-vs-riemann-curvature:RBismut-sym}, will show that $\Alt_{12}(\nabla^gH)=0$ and this will imply, together with $\Bi(\nabla^gH)=0$, that $\nabla^gH=0$. Altogether, this proves \ref{prop:generalized-Riemann-flat:F-isotropic}, \ref{prop:generalized-Riemann-flat:H-parallel}, \ref{prop:generalized-Riemann-flat:Rpm} and \ref{prop:generalized-Riemann-flat:Rg} for \ref{prop:generalized-Riemann-flat:levi-civita}. Notice that in this case we do not necessarily have $\Bi(H^2)=0$.

Lastly, the results on sectional and scalar curvature follow from the same argument as in Proposition \ref{prop:generalized-Riemann-flat}.
\end{proof}

By Proposition \ref{prop:generalized-Riemann-flat:app}, when $(1+2\lambda^+)(1-3\lambda^-)=0$, there could possibly be flat generalized Bismut connections not satisfying \ref{prop:generalized-Riemann-flat:F-isotropic}-\ref{prop:generalized-Riemann-flat:scal}. For $\lambda^+=-\frac{1}{2}$, it can be checked that such a connection would exist if and only if we can find $H$, $F$ and $g$ such that $F^2$ is a $4$-form and
\begin{align*}
\nabla^gH &= \frac{3}{2}F^2, &  \Bi(H^2)&=-\frac{1}{2}dH, &
R^g &= \frac{1}{4}(F^2-H^2).
\end{align*}
For $\lambda^-=\frac{1}{3}$ and $\lambda^+\neq -\frac{1}{3}$, such a connection would exist if and only if we can find $H$, $F$ and $g$ such that $F^2$ is a 4-form and.
\begin{align*}
\nabla^gH &= \frac{3}{2}F^2, & \Bi(H^2)&=\frac{2+3\lambda^+}{2(1+3\lambda^+)}dH, &
R^g &= \frac{1}{4}\left( 1+\frac{1}{1+3\lambda^+}\right) F^2-\frac{1}{4}H^2.
\end{align*}
This is an interesting question, since it leaves open the possibility of more distinguished points, or  generally subsets, in the Bismut family when the metric is pseudo-Riemannian.

\bigskip
\section{Computations} \label{app:computations}

We collect here the proofs of Lemma \ref{lem:brackets-in-regular-ca} and Proposition \ref{prop:curvatures-regular-CA}.

\begin{proof}[Proof of Lemma \ref{lem:brackets-in-regular-ca}]
Define the operator $\nabla^\pm_XY:=\rho([X^\mp,Y^\pm]_\pm)$. We will see that these are actually the Bismut connections of $(g,H)$. First of all, it is clear that they are connections on $TM$. Moreover, using that $\prodesc{a_\pm}{b_\pm}=\pm g(\rho a_\pm, \rho b_\pm)$ for $a_\pm,b_\pm\in V_\pm$, we see that they are metric connections:
\begin{align*}
\Lcal_X(g(Y,Z)) &=\pm \Lcal_{\rho X^\mp}\prodesc{Y^\pm}{Z^\pm}=\pm\prodesc{[X^\mp,Y^\pm]_\pm}{Z^\pm}\pm\prodesc{Y^\pm}{[X^\mp,Z^\pm]_\pm} \\
&=g(\nabla^\pm_XY,Z)+g(Y,\nabla^\pm_XZ).
\end{align*}
We can also compute their torsion:
\begin{align*}
g(T^{\nabla^\pm}(X,Y),Z) &= \pm\prodesc{[X^\mp,Y^\pm]_\pm-[Y^\mp,X^\pm]_\pm-[X,Y]^\pm}{Z^\pm} \\
&=\pm\prodesc{[X^\pm-X^\mp,Y^\mp-Y^\pm]}{Z^\pm} \\
&\qquad \pm\prodesc{[X^\mp,Y^\mp]+[X^\pm,Y^\pm]-[X,Y]^\pm}{Z^\pm}.
\end{align*}
Since $[\cdot,\cdot]$ vanishes on 1-forms, the first term is zero. For the second term we use \eqref{eq:bracket} to see that
\begin{align*}
g(T^{\nabla^\pm}(X,Y),Z) &=\pm\prodesc{[X,Y]\mp g[X,Y]+2i_Yi_XH+2F(X,Y)}{Z^\pm} \\
&=\pm 2\prodesc{i_Yi_XH}{Z^\pm}=\pm H(X,Y,Z).
\end{align*}
Therefore, $\nabla^\pm$ are indeed the Bismut connections of $(g,H)$.

From this we can conclude, using \eqref{eq:bracket} again, that
\[
[X^\mp,Y^\pm]_\pm=(\nabla^\pm_XY)^\pm+\frac{1\mp 1}{2}F(X,Y).
\]
Since $[X^\pm,Y^\mp]=-[Y^\mp,X^\pm]$, we finally have that
\begin{align*}
[X^\pm,Y^\mp] &= [X^\pm,Y^\mp]_\mp-[Y^\mp,X^\pm]_\pm \\
&= (\nabla^\mp_XY)^\mp+\frac{1\pm 1}{2}F(X,Y)-(\nabla^\pm_YX)^\pm -\frac{1\mp 1}{2}F(Y,X) \\
&=(\nabla^\mp_XY)^\mp-(\nabla^\pm_YX)^\pm+F(X,Y).
\end{align*}

On the other hand, we have that
\[
[X^\pm, Y^\pm]=[X\pm gX,Y\pm gY]=[X,Y]\pm (\Lcal_X(gY)-i_Yd(gX)\pm i_Yi_XH)+F(X,Y).
\]
Using the Koszul formula for the Levi-Civita connection of $g$ one can compute
\[
g^{-1}(\Lcal_X(gY)-i_Yd(gX)\pm i_Yi_XH)=2(\nabla^\pm X)^*Y+[X,Y]
\]
This gives, then, that
\begin{align*}
[X^\pm,Y^\pm]_\pm &=\frac{1}{2}\big( [X,Y]+2(\nabla^\pm X)^*Y+[X,Y] \pm g([X,Y]+2(\nabla^\pm X)^*Y+[X,Y]) \big) \\
&\qquad\qquad +\frac{1\mp 1}{2}F(X,Y) \\
&=((\nabla^\pm X)^*Y+[X,Y])^\pm+\frac{1\mp 1}{2}F(X,Y).
\end{align*}
and that
\begin{align*}
[X^\pm,Y^\pm]_\mp &=\frac{1}{2}\big( [X,Y]-2(\nabla^\pm X)^*Y-[X,Y] \mp g([X,Y]-2(\nabla^\pm X)^*Y-[X,Y] ) \big) \\
&\qquad\qquad +\frac{1\pm 1}{2} F(X,Y)\\
&=-((\nabla^\pm X)^*Y)^\mp+\frac{1\pm 1}{2} F(X,Y).
\end{align*}
From here we finally conclude that
\[
[X^\pm,Y^\pm]=((\nabla^\pm X)^*Y+[X,Y])^\pm-((\nabla^\pm X)^*Y)^\mp + F(X,Y).
\]

Lastly, \eqref{eq:bracket} gives
\[
[X^\pm,r]=[X\pm gX,r]=-\prodesc{i_XF}{r}+\nabla_Xr. \qedhere
\]
\end{proof}

\begin{proof}[Proof of Proposition \ref{prop:curvatures-regular-CA}]
Throughout the proof we will use the notation $\nablabarpm=\nabla^\pm+B^\pm$. 

Let us start with $\Rcal^+_D$. Using Lemma \ref{lem:brackets-in-regular-ca} and that $\nabla^\mp_XY-\nabla^\pm_YX=[X,Y]$,
\begin{align*}
\Rcal^+_D(X^+,Y^-)Z^+ &= D_{X^+}(\nabla^+_YZ)^+ - D_{Y^-}(\nablabarp_XZ)^+ - D_{[X^+,Y^-]}Z^+ \\
&=\left( \nablabarp_X\nabla^+_YZ-\nabla^+_Y\nablabarp_XZ-\nabla^+_{\nabla^-_XY}Z+\nablabarp_{\nabla^+_YX}Z -\frac{1}{2} F_{F(X,Y)}Z \right)^+  \\
&=\left( R^+(X,Y)Z-(\nabla^+_YB^+)_XZ -\frac{1}{2} F_{F(X,Y)}Z \right)^+.
\end{align*}

For the second one, notice that if $\alpha\in\Omega^1(M)$, then
\[
D_\alpha Z^+=\frac{1}{2}D_{\alpha^+-\alpha^-}Z^+=\frac{1}{2}\big( \nablabarp_{g^{-1}\alpha}Z-\nabla^+_{g^{-1}\alpha}Z \big)^+=\frac{1}{2}(B^+_{g^{-1}\alpha}Z)^+,
\]
so
\begin{align*}
\Rcal^+_D(X^+,r)Z^+ &=\frac{1}{2}D_{X^+}(F_r Z)^+-D_r(\nablabarp_XZ)^+-D_{[X^+,r]}Z^+ \\
&=\left( \frac{1}{2}\nablabarp_X(F_r Z) - \frac{1}{2}F_r (\nablabarp_XZ) +\frac{1}{2}B^+_{F_r X}Z-\frac{1}{2}F_{\nabla_Xr}Z \right)^+ \\
&=\frac{1}{2}\left( (\nabla^+_XF)_r Z+B^+_X(F_r Z)-F_r B^+_XZ+B^+_{F_r X}Z \right)^+.
\end{align*}

We now turn to $\Rcal^-_D$:
\begin{align*}
\Rcal^-_D(X^-,Y^+)Z^- &= D_{X^-}((\nabla^-_YZ)^-+F(Y,Z))-D_{Y^+}((\nablabarm_XZ)^-\\
&\qquad\qquad +A(X,Z))-D_{[X^-,Y^+]}Z^- \\
&=( \nablabarm_X\nabla^-_YZ)^-+A(X,\nabla^-_YZ)+(A_{F(Y,Z)}X)^-+\nabla_X(F(Y,Z)) \\
&\qquad\qquad +N_X(F(Y,Z))-(\nabla^-_Y\nablabarm_XZ)^--F(Y,\nablabarm_XZ)+\frac{1}{2}(F_{A(X,Z)}Y)^-\\
&\qquad\qquad -\nabla_Y(A(X,Z))-(\nabla^-_{\nabla^+_XY}Z)^--F(\nabla^+_XY,Z)+(\nablabarm_{\nabla^-_YX}Z)^- \\
&\qquad\qquad +A(\nabla^-_YX,Z)-(W_{F(X,Y)}Z)^--C_Z(F(X,Y)) \\
&= \left(R^-(X,Y)Z-(\nabla^-_YB^-)_XZ + A_{F(Y,Z)}X \right. \\
&\qquad\qquad \left.-W_{F(X,Y)}Z+\frac{1}{2}F_{A(X,Z)}Y \right)^- \\
&\qquad\qquad -(\nabla^-_YA)(X,Z)+(\nabla^-_XF)(Y,Z)-F(Y,B^-_XZ)\\
&\qquad\qquad -F(H(X,Y),Z)-C_Z(F(X,Y)).
\end{align*}
\begin{align*}
\Rcal^-_D(X^-,Y^+)r &=D_{X^-}\left( -\frac{1}{2}(F_r Y)^-+\nabla_Yr) \right)-D_{Y^+}((A_r X)^-+\nabla_Xr+N_Xr)\\
&\qquad\qquad -D_{[X^-,Y^+]}r \\
&=-\frac{1}{2}(\nablabarm_X(F_r Y))^--\frac{1}{2}A(X,F_r Y)+(A_{\nabla_Yr}X)^-+\nabla_X\nabla_Yr\\
&\qquad\qquad +N_X(\nabla_Yr)-(\nabla^-_Y(A_r X))^--F(Y,A_r X)+\frac{1}{2}(F_{\nabla_Xr+N_Xr}Y))^-\\
&\qquad\qquad -\nabla_Y(\nabla_Xr+N_Xr)+\frac{1}{2}(F_r(\nabla^+_XY))^--\nabla_{\nabla^+_XY}r+(A_r(\nabla^-_YX))^-\\
&\qquad\qquad +\nabla_{\nabla^-_XY}r+N_{\nabla^-_YX}r-C(F(X,Y),r)^--L_{F(X,Y)}r \\
&=\left(-\frac{1}{2}(\nabla^-_XF)_r Y-(\nabla^-_YA)_r X-\frac{1}{2}B^-_XF_r Y \right. \\
&\qquad\qquad \left. +\frac{1}{2}F_r(H(X,Y))-C(F(X,Y),r) +\frac{1}{2}F_{N_Xr}Y\right)^- \\
&\qquad\qquad +R^\nabla(X,Y)r -(\nabla^-_YN)_Xr-\frac{1}{2}A(X,F_r Y) \\
&\qquad\qquad -F(Y,A_r X)+L_{F(X,Y)}r.
\end{align*}
For the next one, notice that if $\alpha\in\Omega^1(M)$, then
\begin{align*}
D_\alpha Z^-&=\frac{1}{2}D_{\alpha^+-\alpha^-}Z^-\\
&=\frac{1}{2}\left( \big( \nabla^-_{g^{-1}\alpha}Z-\nablabarm_{g^{-1}\alpha}Z\big)^-+F(g^{-1}\alpha,Z)-A(g^{-1}\alpha,Z) \right)\\
&=\frac{1}{2}( -(B^-_{g^{-1}\alpha}Z)^-+F(g^{-1}\alpha,Z)-A(g^{-1}\alpha,Z)),
\end{align*}
so
\begin{align*}
\Rcal^-_D(r,Y^+)Z^- &=D_r ((\nabla^-_YZ)^-+F(Y,Z))-D_{Y^+}((W_r Z)^-+C_Zr)-D_{[r,Y^+]}Z^- \\
&=(W_r(\nabla^-_YZ))^-+C_{\nabla^-_YZ}r+C(r,F(Y,Z))^-+L_r(F(Y,Z)) \\
&\qquad\qquad -(\nabla^-_Y(W_r Z))^--F(Y,W_r Z)+\frac{1}{2}(F_{C_Zr}Y)^--\nabla_Y(C_Zr) \\
&\qquad\qquad +\frac{1}{2}(B^-_{F_r Y}Z)^--\frac{1}{2}F(F_r Y,Z)+\frac{1}{2}A(F_r Y,Z) \\
&\qquad\qquad +(W_{\nabla_Yr}Z)^-+C_Z(\nabla_Yr) \\
&=\left(-(\nabla^-_YW)_r Z+C(r,F(Y,Z))+\frac{1}{2}F_{C_Zr}Y-\frac{1}{2}B^-_{F_r Y}Z
\right)^- \\
&\qquad\qquad -(\nabla^-_YC)_Zr+L_r(F(Y,Z))-F(Y,W_r Z)\\
&\qquad\qquad +\frac{1}{2}A(F_r Y, Z)-\frac{1}{2}F(F_r Y,Z).
\end{align*}

Lastly, if $\alpha\in\Omega^1(M)$, then
\[
D_\alpha r=\frac{1}{2}D_{\alpha^+-\alpha^-}r=-\frac{1}{2}\left(\frac{1}{2}F_r(g^{-1}\alpha)+A_r(g^{-1}\alpha) \right)^--N_{g^{-1}\alpha}r,
\]
so
\begin{align*}
\Rcal^-_D(r,Y^+)s &=D_r\left(-\frac{1}{2}(F_s Y)^-+\nabla_Ys\right)-D_{Y^+}(C(r,s)^-+L_rs) -D_{[r,Y^+]}s \\
&=-\frac{1}{2}(W_r F_s Y)^--\frac{1}{2}C_{F_s Y}r+C(r,\nabla_Ys)^-+L_r(\nabla_Ys) \\
&\qquad -(\nabla^-_Y(C(r,s)))^--F(Y,C(r,s))+\frac{1}{2}(F_{L_rs}Y)^--\nabla_Y(L_rs) \\
&\qquad +\frac{1}{2}\left( \frac{1}{2}F_s F_r Y+A_s F_r Y \right)^-+N_{F_r Y}s+C(\nabla_Yr,s)^-+L_{\nabla_Yr}s \\
&=\left( -\nabla^-_YC(r,s)-\frac{1}{2}W_r F_s Y+\frac{1}{2}F_{L_rs}Y+\frac{1}{4}F_s F_r Y +\frac{1}{2}A_s F_r Y\right)^- \\
&\qquad -(\nabla_YL)_rs -\frac{1}{2}C_{F_s Y}r-F(Y,C(r,s))+N_{F_r Y}s. \qedhere
\end{align*}
\end{proof}

\bibliographystyle{alpha}
\bibliography{Flat-biblio}

\end{document}

%% file: Flat-commands.tex
\newcommand{\R}{\mathbb{R}} 

\newcommand{\scal}{\operatorname{scal}} 
\renewcommand{\sec}{\operatorname{sec}} 
\newcommand{\Ad}{\operatorname{Ad}} 
\newcommand{\Hom}{\operatorname{Hom}} 
\newcommand{\Alt}{\operatorname{Alt}} 
\newcommand{\Bi}{\operatorname{Bi}} 

\newcommand{\nablabarpm}{\nablabar\hspace{0pt}^\pm} 
\newcommand{\nablabarp}{\nablabar\hspace{0pt}^+} 
\newcommand{\nablabarm}{\nablabar\hspace{0pt}^-} 


\newcommand{\norm}[1]{\Vert #1 \Vert} 
\newcommand{\prodesc}[2]{\langle #1,#2 \rangle} 



\newcommand{\End}{\operatorname{End}}


\newcommand{\gfrak}{\mathfrak{g}}

\newcommand{\pfrak}{\mathfrak{p}}
\newcommand{\kfrak}{\mathfrak{k}}

\newcommand{\Xfrak}{\mathfrak{X}}

\newcommand{\so}{\mathfrak{so}}

\newcommand{\Lcal}{\mathcal{L}}

\newcommand{\Rcal}{\mathcal{R}}
\newcommand{\Gcal}{\mathcal{G}}

\newcommand{\Sbb}{\mathbb{S}} 
\newcommand{\Tbb}{\mathbb{T}}

\newcommand{\Cinf}{C^\infty}

\newcommand{\nablabar}{\overline{\nabla}}




